\documentclass[12pt]{amsart}

\usepackage[T1]{fontenc}
\usepackage[utf8]{inputenc}
\usepackage{lmodern}

\usepackage[american]{babel}
\usepackage[babel, final]{microtype}
\usepackage{amssymb}
\usepackage[all]{xy}

\usepackage{ifthen}
\usepackage{ifpdf}
\usepackage{multirow}
\usepackage[shortlabels]{enumitem}

\ifpdf
  \usepackage[pdftex, dvipsnames]{xcolor}
  \usepackage[pdftex, final]{graphicx}
  \usepackage{caption}
  \usepackage{subcaption}
  \usepackage{pinlabel}
  \usepackage[pdftex,%
    letterpaper,%
    includehead,%
    includefoot,%
    nomarginpar,%
    lmargin=1in,%
    rmargin=1in,%
    tmargin=1in,%
    bmargin=1in,%
  ]{geometry}
  \usepackage[pdftex,%
    final,%
    colorlinks=true,%
    linkcolor=NavyBlue,%
    citecolor=NavyBlue,%
    filecolor=NavyBlue,%
    menucolor=NavyBlue,%
    urlcolor=NavyBlue,%
    bookmarks=true,%
    bookmarksdepth=3,%
    bookmarksnumbered=true,%
    bookmarksopen=true,%
    bookmarksopenlevel=2,%
  ]{hyperref}
  \hypersetup{
    pdftitle={Spectral GRID invariants and Lagrangian cobordisms},
    pdfauthor={Mitchell Jubeir, Ina Petkova, Noah Schwartz, Zachary Winkeler, 
      C.-M. Michael Wong},
    pdfsubject={Lagrangian cobordisms},
    pdfkeywords={Lagrangian cobordism, Legendrian knots, Heegaard Floer 
      homology}
  }
\else
  \usepackage[dvips, dvipsnames]{xcolor}
  \usepackage[dvips, final]{graphicx}
  \usepackage{caption}
  \usepackage{subcaption}
  \usepackage{pinlabel}
  \usepackage[dvips,%
    letterpaper,%
    includehead,%
    includefoot,%
    nomarginpar,%
    lmargin=1in,%
    rmargin=1in,%
    tmargin=1in,%
    bmargin=1in,%
  ]{geometry}
  \usepackage[ps2pdf,%
    final,%
    colorlinks=true,%
    linkcolor=NavyBlue,%
    citecolor=NavyBlue,%
    filecolor=NavyBlue,%
    menucolor=NavyBlue,%
    urlcolor=NavyBlue,%
    bookmarks=true,%
    bookmarksdepth=3,%
    bookmarksnumbered=true,%
    bookmarksopen=true,%
    bookmarksopenlevel=2,%
  ]{hyperref}
  \hypersetup{
    pdftitle={Spectral GRID invariants and Lagrangian cobordisms},
    pdfauthor={Mitchell Jubeir, Ina Petkova, Noah Schwartz, Zachary Winkeler, 
      C.-M. Michael Wong},
    pdfsubject={Lagrangian cobordisms},
    pdfkeywords={Lagrangian cobordism, Legendrian knots, Heegaard Floer 
      homology}
  }
\fi

\usepackage{mathdots}
\usepackage{float}
\usepackage[mathscr]{euscript}
\usepackage[normalem]{ulem}
\usepackage{tikz}
\usepackage{tikz-cd}
\usepackage{mathtools}
\usepackage{stmaryrd}
\SetSymbolFont{stmry}{bold}{U}{stmry}{m}{n}
\usetikzlibrary{arrows,automata,fit}
\usetikzlibrary{matrix,arrows}
\usetikzlibrary{decorations.pathreplacing,calligraphy}

\input{macros}

\newcommand{\set}[1]{\mathchoice%
  {\left\lbrace #1 \right\rbrace}%
  {\lbrace #1 \rbrace}%
  {\lbrace #1 \rbrace}%
  {\lbrace #1 \rbrace}%
}
\newcommand{\setc}[2]{\mathchoice%
  {\left\lbrace #1 \, \middle\vert \, #2 \right\rbrace}%
  {\lbrace #1 \, \vert \, #2 \rbrace}%
  {\lbrace #1 \, \vert \, #2 \rbrace}%
  {\lbrace #1 \, \vert \, #2 \rbrace}%
}

\newcommand{\abs}[1]{\mathchoice%
  {\left\lvert #1 \right\rvert}%
  {\lvert #1 \rvert}%
  {\lvert #1 \rvert}%
  {\lvert #1 \rvert}%
}

\newcommand{\dirsum}{\oplus}                 
\newcommand{\tensor}{\otimes}                
\newcommand{\union}{\cup}                    
\newcommand{\intersect}{\cap}                

\newcommand{\cross}{\times}                  

\newcommand{\isom}{\cong}                    

\newcommand{\numset}[1]{\mathbb{#1}}
\newcommand{\N}{\numset{N}}
\newcommand{\Z}{\numset{Z}}

\newcommand{\R}{\numset{R}}

\newcommand{\F}[1]{\numset{F}_{#1}}

\newcommand{\bdy}{\partial} 

\newcommand{\xistd}{\xi_{\mathrm{std}}}
\DeclareMathOperator{\tb}{tb}
\DeclareMathOperator{\rot}{r}

\DeclareMathOperator{\HF}{HF}

\DeclareMathOperator{\CFL}{CFL}
\DeclareMathOperator{\HFL}{HFL}

\newcommand{\opminus}[1]{#1^-}

\newcommand{\ophat}[1]{\widehat{#1}}
\newcommand{\optilde}[1]{\widetilde{#1}}

\newcommand{\HFh}{\ophat{\HF}}

\newcommand{\CFLh}{\ophat{\CFL}}
\newcommand{\CFLt}{\optilde{\CFL}}

\newcommand{\HFLh}{\ophat{\HFL}}
\newcommand{\HFLt}{\optilde{\HFL}}

\newcommand{\curves}[1]{\boldsymbol{#1}}
\newcommand{\alphas}[1][]{%
  \ifthenelse{\equal{#1}{}}{\curves{\alpha}}{\curves{\alpha^{#1}}}
}
\newcommand{\betas}[1][]{%
  \ifthenelse{\equal{#1}{}}{\curves{\beta}}{\curves{\beta_{#1}}}
}

\newcommand{\gen}[1]{\mathbf{#1}}
\newcommand{\genset}[1]{\mathfrak{#1}}
\newcommand{\x}{\gen{x}}
\newcommand{\y}{\gen{y}}

\DeclareMathOperator{\gr}{gr}

\newcommand{\emptypoly}[1]{#1^\circ}

\newcommand{\grid}{\mathbb{G}}

\DeclareMathOperator{\GC}{GC}
\DeclareMathOperator{\GH}{GH}

\newcommand{\GCm}{\opminus{\GC}}

\newcommand{\GCt}{\optilde{\GC}}

\newcommand{\GHt}{\optilde{\GH}}

\newcommand{\markers}[1]{\mathbb{#1}}
\newcommand{\OO}{\markers{O}}
\newcommand{\XX}{\markers{X}}

\newcommand{\eRect}{\emptypoly{\Rect}}

\newcommand{\ePent}{\emptypoly{\Pent}}

\newcommand*{\Ftwo}{\F{2}}

\DeclareMathOperator{\Int}{Int}

\newcommand*{\mir}[1]{m (#1)}

\newcommand*{\alphastd}{\alpha_{\mathrm{std}}}

\renewcommand*{\tb}{\mathit{tb}}
\renewcommand*{\rot}{\mathit{r}}

\newcommand*{\gCFLh}{\mathrm{g}\CFLh}

\newcommand*{\fGC}{\mathcal{GC}}

\newcommand*{\fGCt}{\optilde{\fGC}}

\newcommand*{\ssconv}{\rightrightarrows}

\newcommand*{\Legm}{\Leg_-}
\newcommand*{\Legp}{\Leg_+}
\newcommand*{\Lag}{L}

\newcommand*{\lambdat}{\optilde{\lambda}}
\newcommand*{\lambdatp}{\lambdat^+}
\newcommand*{\lambdatm}{\lambdat^-}
\newcommand*{\lambdatpm}{\lambdat^\pm}

\newcommand*{\stabmap}{\mathcal{S}}
\newcommand*{\destabmap}{\mathcal{D}}
\newcommand*{\commmap}{\mathcal{C}}
\newcommand*{\pinchmap}{\mathcal{P}}
\newcommand*{\birthmap}{\mathcal{B}}

\renewcommand*{\genset}[1]{S (#1)}

\newcommand*{\disjunion}{\sqcup}

\newcommand*{\AB}{\mathit{AB}}
\newcommand*{\NB}{\mathit{NB}}
\newcommand*{\AN}{\mathit{AN}}
\newcommand*{\NN}{\mathit{NN}}

\newcommand{\J}{\mathcal{J}}
\newcommand{\I}{\mathcal{I}}

\newcommand{\X}{\mathbb{X}}
\newcommand{\G}{\mathbb{G}}
\newcommand{\filt}{\mathcal{F}}
\newcommand{\iso}{\cong}

\newcommand{\diff}{\partial}
\DeclareMathOperator{\grGH}{\widetilde{GH}}
\newcommand{\SG}{S(\G)}
\newcommand{\SGp}{S(\G_+)}

\newcommand{\Leg}{\Lambda}

\newcommand{\xp}{\mathbf{x}^+}
\newcommand{\xm}{\mathbf{x}^-}
\newcommand{\xpm}{\mathbf{x}^\pm}
\newcommand{\lp}{\lambda^+}
\newcommand{\lm}{\lambda^-}

\newcommand{\np}{n^+}
\newcommand{\nm}{n^-}
\newcommand{\npm}{n^\pm}
\newcommand{\std}{\mathrm{std}}
\newcommand{\dmi}{\widetilde{\bdy}_{\OO}}
\newcommand{\dpl}{\widetilde{\bdy}_{\OO}}

\newcommand{\bdt}{\widetilde{\partial}_{\OO}}
\DeclareMathOperator{\Rect}{Rect}
\DeclareMathOperator{\Tri}{Tri}
\DeclareMathOperator{\Pent}{Pent}

\newcommand*{\fC}{\mathcal{C}}

\newcommand*{\cref}{\fullfref}

\newcommand*{\extder}[1]{#1}

\title[Spectral GRID invariants and Lagrangian cobordisms]{Spectral GRID 
  invariants and Lagrangian cobordisms}

\author[M. Jubeir]{Mitchell Jubeir}
\address{Department of Mathematics \\ UC Santa Barbara \\ Santa Barbara, CA 93106}
\email{\href{mailto:mitchell\_jubeir@ucsb.edu}{mitchell\_jubeir@ucsb.edu}}

\author[I. Petkova]{Ina Petkova}
\address{Department of Mathematics \\ Dartmouth College \\ Hanover, NH 03755}
\email{\href{mailto:ina.petkova@dartmouth.edu}{ina.petkova@dartmouth.edu}}
\urladdr{\url{https://math.dartmouth.edu/~ina/}}

\author[N. Schwartz]{Noah Schwartz}
\address{Johns Hopkins University Applied Physics Laboratory \\ Laurel, MD 
  20723}
\email{\href{mailto:Noah.Schwartz@jhuapl.edu}{Noah.Schwartz@jhuapl.edu}}

\author[Z. Winkeler]{Zachary Winkeler}
\address{Department of Mathematical Sciences \\ Smith College \\ Northampton, 
  MA 01063}
\email{\href{mailto:zwinkeler@smith.edu}{zwinkeler@smith.edu}}
\urladdr{\url{https://zach-winkeler.github.io/}}

\author[C.-M. M. Wong]{C.-M. Michael Wong}
\address{Department of Mathematics and Statistics \\ University of Ottawa \\ 
  Ottawa, ON K1N 6N5}
\email{\href{mailto:Mike.Wong@uOttawa.ca}{Mike.Wong@uOttawa.ca}}
\urladdr{\url{https://mysite.science.uottawa.ca/cwong/}}

\begin{document}

\begin{abstract}
  We prove that the filtered GRID invariants of Legendrian links in link Floer 
  homology, and consequently their associated invariants in the spectral 
  sequence, obstruct decomposable Lagrangian cobordisms in the symplectization 
  of the standard contact structure on $\R^3$, strengthening a result by 
  Baldwin, Lidman, and the fifth author.
\end{abstract}

\maketitle

\section{Introduction}\label{sec:intro}

An interesting and difficult problem in contact and symplectic geometry is to 
decide, given contact $3$-manifolds $(Y_-, \xi_-)$ and $(Y_+, \xi_+)$, and two 
Legendrian links
\[
  \Legm \subset (Y_-, \xi_-), \qquad \Legp \subset (Y_+, \xi_+),
\]
whether $\Legm$ and $\Legp$ are related by an exact Lagrangian cobordism in a 
Weinstein cobordism from $(Y_-, \xi_-)$ to $(Y_+, \xi_+)$. Even in the simplest 
case where $Y_\pm = \R^3$ and $\xi_\pm = \xistd$ is the standard contact 
structure given by the kernel of the $1$-form
\[
  \alphastd = \extder{d}z - y \, \extder{d}x,
\]
and the Weinstein cobordism is the symplectization
\[
  (\R_t \cross \R^3, \extder{d} (e^t \alphastd)),
\]
this turns out to be a challenging problem.

Focusing on this case in the present article, we recall that the classical 
invariants---the Thurston--Bennequin and rotation numbers---do give 
obstructions: Chantraine \cite{Cha10:Leg} shows that if there exists an exact 
Lagrangian cobordism $\Lag \colon \Legm \to \Legp$, then
\[
  \tb (\Legp) - \tb (\Legm) = - \chi (L), \qquad \rot (\Legp) = \rot (\Legm).
\]
For instance, since the undestabilizable Legendrian unknot $\Leg_0$ and the 
undestabilizable Legendrian representative $\Leg_1$ of $m (6_2)$ both have $\tb 
= -1$, we know that any exact Lagrangian cobordism from $\Leg_0$ to $\Leg_1$ 
(or from $\Leg_1$ to $\Leg_0$) must be a concordance, but since $m (6_2)$ is 
not smoothly slice (as it has nonzero signature), no such cobordism can 
exist.\footnote{This contrasts with the situation in smooth topology, where a 
  cobordism exists between any pair of links, but the minimum genus of such a 
  cobordism is difficult to determine.} However, the classical invariants do 
not give a complete answer: For example, using the functoriality \cite{EHK16} 
of Legendrian contact homology from symplectic field theory \cite{EGH00, Che02, 
  BC14}, Chantraine \cite{Cha15:NotSym} shows that for the undestabilizable 
representative $\Leg_2$ of $m (9_{46})$, which has $\tb = -1$ and $\rot = 0$, 
there exists a Lagrangian concordance $\Lag \colon \Leg_0 \to \Leg_2$ but not a 
Lagrangian concordance $\Lag' \colon \Leg_2 \to \Leg_0$.\footnote{In 
  particular, this also illustrates the important fact that exact Lagrangian 
  cobordisms are directed.} Thus, an important goal is to develop 
\emph{effective} obstructions, such as Legendrian contact homology as mentioned 
above, that give obstructions beyond the classical invariants.  For more 
examples of obstructions from symplectic field theory, see \cite{ST13, CDGG15, 
  CNS16, Pan17}.

Another source of effective invariants is knot (Heegaard) Floer homology. In 
earlier work, Baldwin, Lidman, and the fifth author \cite{BLW22} prove that 
the so-called \emph{GRID invariants} in knot Floer homology can effectively 
obstruct \emph{decomposable} Lagrangian cobordisms (of any genus), which are 
cobordisms that can be obtained by concatenating elementary cobordisms 
associated to Legendrian isotopies, pinches, and births, as depicted in 
\fullref{fig:leg-moves}.  Decomposable cobordisms are exact, and most known 
connected exact Lagrangian cobordisms between non-empty,\footnote{If the 
  non-empty condition is dropped, there are examples of non-decomposable exact 
  Lagrangian fillings in \cite{Sauv04} and \cite{Cha13:nondec}.} undestabilizable\footnote{A preprint \cite{DG24} has appeared during the review process of this article, in which non-decomposable exact Lagrangian cobordisms are constructed where both ends are stabilized Legendrian knots.}  Legendrian 
links are known to be decomposable \cite{Cha12:dec}; whether all connected 
exact Lagrangian cobordisms between non-empty Legendrian links are decomposable 
remains a major open problem.  If one restricts to cobordisms of genus zero, 
there is also a body of work \cite{BS18, BS21, GJ19} that shows that knot Floer 
homology gives effective obstructions even when the decomposability 
  assumption is weakened (to regularity) or discarded.

The goal of the present article is to extend the result in \cite{BLW22} to 
the context of \emph{filtered} knot Floer chain complexes, which will then 
imply that certain invariants in the associated spectral sequences also provide 
effective obstructions to exact Lagrangian cobordisms. The existence of these 
invariants is known to experts in knot Floer homology, but their definitions 
have not appeared in the literature thus far; in \fullref{ssec:intro-defn} 
below, we first provide the definitions.

\subsection{The spectral GRID invariants}
\label{ssec:intro-defn}

We briefly recall the invariants of Legendrian links in $(\R^3, \xistd)$ 
defined by Ozsv\'ath, Szab\'o, and Thurston \cite{OST08} using the 
combinatorial grid-diagram formulation \cite{MOS09, MOST07} of knot Floer 
homology; for more details, see \fullref{ssec:prelim-grid} below. For ease of 
exposition, we focus on the \emph{tilde} flavor, which computes a stabilized 
version of knot Floer homology. Let $\grid$ be a grid diagram of size $n$ that 
represents a Legendrian link $\Leg \subset (\R^3, \xistd)$; then there are 
\emph{canonical generators} $\xp (\grid)$ and $\xm (\grid)$, which are cycles 
of the \emph{grid chain complex} $\GCt (\grid)$, and thus give rise to homology 
classes
\[
  \lambdatp (\grid), \lambdatm (\grid) \in \GHt (\grid) \isom \HFLh (-S^3, 
  \Leg) \tensor V^{n-\ell},\footnote{A Legendrian link in $(\R^3, \xistd)$ can 
    be naturally viewed as a Legendrian link in the standard contact $S^3$.  We 
    follow \cite{OST08} and view the Legendrian invariants as living in $\HFLh 
    (S^3, \mir{\Leg}) \isom \HFLh (-S^3, \Leg)$.}
\]
where $V$ is a vector space of dimension $2$, and $\ell$ is the number of 
components of $\Leg$. (On the chain level, $\GCt (\grid)$ corresponds to a 
stabilized version of $\gCFLh (-S^3, \Leg)$, whose homology is $\HFLh (-S^3, 
\Leg)$.) Ozsv\'ath, Szab\'o, and Thurston prove that $\lambdatpm$ are preserved 
by the homomorphisms on $\GHt$ induced by Legendrian isotopy, which implies 
that they are Legendrian isotopy invariants; they are referred to as the 
\emph{GRID invariants}. Similar statements hold for the \emph{minus} flavor of 
the theory.

As explained in \cite{MOST07, OSS15}, one may in fact define a filtered 
chain complex $\fGCt (\grid)$ with the \emph{Alexander filtration}, such that 
the associated graded object of $\fGCt (\grid)$ is $\GCt (\grid)$. This 
corresponds to the filtered chain complex $\CFLh$, whose associated graded 
object is $\gCFLh$. The filtered chain homotopy type of $\fGCt (\grid)$ is an 
invariant of the smooth link type of $\Leg$ and the grid size of $\grid$ (which 
controls the amount of stabilization).
More precisely, $\fGCt (\grid)$ is filtered chain homotopy equivalent to $\CFLh (-S^3, \Leg) \otimes V^{n-\ell}$. (This can be seen by extending \cite[Section~5.7]{OSS15} to the filtered context \cite[Section~13.3]{OSS15} and setting all formal variables to zero.) Noting that $\mathrm{H}_* (\CFLh (-S^3, \Leg))$ is simply $\HFh (-S^3)$, we see that
each page of the associated spectral sequence
\begin{equation}
  \label{eq:ss-grid}
  \GHt (\grid) \ssconv \mathrm{H}_* (\fGCt (\grid))
\end{equation}
is isomorphic to each corresponding page of the spectral sequence
\[
  \HFLh (-S^3, \Leg) \otimes V^{n - \ell} \ssconv \HFh (-S^3) \otimes V^{n - 
    \ell},
\]
a stabilized version of the spectral sequence
\[
  \HFLh (-S^3, \Leg) \ssconv \HFh (-S^3),
\]
which is an invariant of the smooth link type of $\Leg$.

The idea of the spectral GRID invariants is to consider the homology classes 
that $\xpm (\grid)$ represent in the spectral sequence in \eqref{eq:ss-grid}.  
These may be defined recursively. For concreteness, let $(E^i, \delta_i)$ be 
the $i^{\text{th}}$ page in the spectral sequence, such that $\GHt (\grid)$ is 
the $E^1$-page, and let $\lambdatp_1 (\grid) = \lambdatp (\grid)$. Suppose that 
$\delta_i (\lambdatp_i (\grid)) = 0$; then we obtain a homology class 
$\lambdatp_{i+1} (\grid) = [\lambdatp_i (\grid)] \in E^{i+1}$. If at any stage 
$\delta_i (\lambdatp_i (\grid)) \neq 0$, then we may instead record the integer 
$i$, which we will call $n^+ (\grid) \in \Z^+$, and the process terminates. Let 
$n^+ (\grid) = \infty$ if $\delta_i (\lambdatp_i) = 0$ for all $i \geq 1$.  
Considering the invariant $\lambdatm (\grid) = [\xm (\grid)]$, we likewise 
obtain invariants $n^- (\grid)$ and $\lambdatm_i (\grid)$. With some work, that 
$n^\pm$ and $\lambdatpm_i$ are indeed invariants of the Legendrian link $\Leg$ 
can be deduced by extending certain proofs in \cite{OST08}, which we will carry 
out later:

\begin{theorem}
  \label{thm:leg}
  Suppose that $\grid$ and $\grid'$ are two grid diagrams that represent the 
  same Legendrian link $\Leg \subset (S^3, \xistd)$. Then $n^+ (\grid) = n^+ 
  (\grid')$ and $n^- (\grid) = n^- (\grid)$, and there exist filtered chain 
  homomorphisms
  \[
    \Phi \colon \fGCt (\grid) \to \fGCt (\grid'), \qquad
    \Phi' \colon \fGCt (\grid') \to \fGCt (\grid),
  \]
  such that the induced maps
  $\Phi_i \colon E^i (\fGCt (\grid)) \to E^i (\fGCt (\grid'))$
  and
  $\Phi'_i \colon E^i (\fGCt (\grid')) \to E^i (\fGCt (\grid))$
  on the pages of the associated spectral sequences satisfy
  \[
    \Phi_i (\lambdatpm_i (\grid)) = \lambdatpm_i (\grid'), \qquad
    \Phi'_i (\lambdatpm_i (\grid)) = \lambdatpm_i (\grid'),
  \]
  for each $1 \leq i \leq n^\pm (\grid)$.
  The fact that $n^\pm (\grid) = n^\pm (\grid')$ means that we have invariants $n^\pm (\Leg) \in \Z^+ \union \set{\infty}$ of the Legendrian link type of $\Leg$.
  Moreover, since the (non)vanishing of $\lambdatpm_i (\grid)$ depends only on $\Leg$, we may view these also as invariants of the Legendrian link type of $\Leg$, and slightly abuse notation to denote them by $\lambdatpm_i (\Leg) \in E^i (\CFLt (-S^3, \Leg))$.
\end{theorem}

\begin{remark}
  \label{rmk:previous-phi}
  Versions of the maps $\Phi$ and $\Phi'$ first appear in \cite[Lemmas~6.5 and 
  6.6]{OST08}, and special cases of \fullref{thm:leg} are spelled out in 
  \cite{NOT08} and \cite{OSS15}.  Specifically, \cite[Theorem~2]{NOT08} 
  distinguishes two Legendrian links $\Leg$ and $\Leg'$ with the same smooth 
  type and classical invariants by showing that $n^+ (\Leg) = 1$ while $n^+ 
  (\Leg') > 1$.
  For technical reasons, our formulation of $\Phi$ and $\Phi'$ differs slightly 
  from \cite{OST08, OSS15} when grid (de)stabilizations are involved, and our 
  language in homological algebra also differs from \cite[Chapter~14]{OSS15}.
\end{remark}

\begin{remark}
  \label{rmk:kmvw}
  The idea to obtain refined information about a contact-geometric invariant in 
  knot Floer homology by considering its homology classes in a spectral 
  sequence is also found in \cite{KMVW19}. However, the filtration (and hence the 
  spectral sequence) that we consider in the present article is distinct from 
  that in \cite{KMVW19}.
\end{remark}

\begin{remark}
  \label{rmk:extensions}
  As shown in \cite{OST08}, the invariant $\lambdatp (\Leg)$ is preserved under 
  negative stabilizations of Legendrian links, implying that it is an invariant 
  of the transverse push-off of $\Leg$. Similarly, $\lambdatp_i$ can also be 
  shown to be transverse invariants by a minor extension of the proof of 
  \fullref{thm:leg}, even though we do not pursue this further in the present 
  article.

  Moreover, there is a version of \fullref{thm:leg} for the \emph{minus} flavor 
  of knot Floer homology, whose proof is also similar. For simplicity, we will 
  work solely with the \emph{tilde} flavor.
\end{remark}

\subsection{Obstructions}
\label{ssec:intro-obstructions}

Baldwin, Lidman, and the fifth author \cite{BLW22} prove that, if there 
exists a decomposable Lagrangian cobordism $L \colon \Legm \to \Legp$, then 
there exists a homomorphism $\Phi_L \colon \HFLt (-S^3, \Legp) \to \HFLt (-S^3, 
\Legm)$ that sends $\lambdatpm (\Legp)$ to $\lambdatpm (\Legm)$. In particular:

\begin{theorem}[{\cite[Theorem~1.2]{BLW22}}]
  \label{thm:BLW22}
  Suppose that $\Leg_-$ and $\Leg_+$ are Legendrian links in $(\R^3, \xistd)$, 
  such that
  \begin{itemize}
    \item $\lambdatp(\Leg_+)=0$ and $\lambdatp(\Leg_-)\neq 0$; or
    \item $\lambdatm(\Leg_+)=0$ and $\lambdatm(\Leg_-)\neq 0$.
  \end{itemize}
  Then there does not exist a decomposable Lagrangian cobordism from $\Leg_-$ 
  to $\Leg_+$.
\end{theorem}

As mentioned above, the main goal of this article is to prove the following 
extension of this result:

\begin{theorem}
  \label{thm:cob}
  Suppose that $\Leg_-$ and $\Leg_+$ are Legendrian links in $(\R^3, \xistd)$, 
  such that
  \begin{itemize}
    \item $\np(\Leg_+) > \np(\Leg_-)$; or
    \item $\lambdatp_i(\Leg_+)=0$ and $\lambdatp_i(\Leg_-)\neq 0$ for some $1 
      \leq i \leq \min\{\np(\Leg_-), \np(\Leg_+)\}$; or
    \item $\nm(\Leg_+) > \nm(\Leg_-)$; or
    \item $\lambdatm_i(\Leg_+)=0$ and $\lambdatm_i(\Leg_-)\neq 0$ for some $1 
      \leq i \leq \min\{\np(\Leg_-), \np(\Leg_+)\}$.
  \end{itemize}
  Then there does not exist a decomposable Lagrangian cobordism from $\Leg_-$ 
  to $\Leg_+$.
\end{theorem}

\begin{remark}
  \label{rmk:antecedents}
  Restricting to decomposable Lagrangian concordances (i.e.\ the case that the 
  genus $g$ is zero), the hypotheses involving $\lambdatp_1$ are covered by 
  \cite{BS18, BS21} and \cite[Corollary~1.4]{GJ19}.  In addition, for $g = 
  0$, \cite[Corollary~1.5]{GJ19} covers the hypothesis $\np (\Legp) > \np 
  (\Legm) = 1$.
\end{remark}

Specializing to the case where $\Legm$ is the undestabilizable Legendrian 
unknot, we obtain the following:

\begin{corollary}
  \label{cor:filling}
  Suppose that $\Leg$ is a Legendrian link in $(\R^3, \xistd)$, such that 
  $\lambdatp_i(\Leg) = 0$ for some $1 \leq i \leq \np (\Leg)$, or $\lambdatm_i 
  (\Leg) = 0$ for some $1 \leq i \leq \nm (\Leg)$.  Then there does not exist a 
  decomposable Lagrangian filling of $\Leg$.
\end{corollary}

\subsection{Functoriality}
\label{ssec:intro-proof}

Analogous to \cite{BLW22}, \fullref{thm:cob} follows from the following theorem, which states that the spectral GRID invariants satisfy a weak functoriality under decomposable Lagrangian cobordisms, in the style of \fullref{thm:leg}.

\begin{theorem}
  \label{thm:functoriality}
  Suppose that $\grid_-$ and $\grid_+$ are two grid diagrams that represent Legendrian links $\Legm$ and $\Legp$ in $(\R^3, \xistd)$ respectively. Suppose that there exists a decomposable Lagrangian cobordism $\Lag$ from $\Legm$ to $\Legp$. Then there exists a filtered chain homomorphism
  \[
    \Psi \colon \fGCt (\grid_+) \to \fGCt (\grid_-) \left\llbracket - \chi (\Lag), \frac{\abs{\Leg_+} - \abs{\Leg_-} - \chi (\Lag)}{2} \right\rrbracket,
  \]
  such that the induced maps
  \(
    \Psi_i \colon E^i (\fGCt (\grid_+)) \to E^i (\fGCt (\grid_-))
  \)
  on the pages of the associated spectral sequences satisfy
  \[
    \Psi_i (\lambdatpm_i (\grid_+)) = \lambdatpm_i (\grid_-),
  \]
  for each $1 \leq i \leq n^\pm (\grid)$. Here, $\abs{\Leg_\pm}$ is the number of components of $\Leg_\pm$.
\end{theorem}

Throughout the paper, we denote by $\fC \llbracket M, A \rrbracket$ the filtered chain complex obtained from $\fC$ by shifting the homological (Maslov) grading up by $M$ and shifting the (Alexander) filtration level up by $A$; in other words, $\fC \llbracket M, A \rrbracket_{(0, 0)} = \fC_{(-M, -A)}$.

As in \cite{BLW22}, we believe but do not prove $\Psi$ to be the \emph{functorial} map of Zemke \cite{Zem19} associated to a certain decorated link cobordism.

\subsection{Computation}
\label{ssec:intro-computation}

The invariants $\npm$ and $\lambdatpm_i$ can be computed algorithmically and 
directly from the filtered chain complex $\fGCt (\grid)$, without the need to 
compute each page of the spectral sequence.  Using some homological algebra, we 
prove in \fullref{sec:computation} the following proposition:

\begin{proposition}
  \label{prop:intro-computation}
  Let $\grid$ be a grid diagram, and let $A$ be the Alexander filtration level 
  of $\xp (\grid)$.  Then
  \begin{itemize}
    \item $\np (\grid) = i < \infty$ if and only if $i$ is the smallest number 
      such that
      \begin{equation}
        \label{eqn:computation-npm}
        [\widetilde{\bdy} \xp (\grid)] \neq 0 \in \mathrm{H}_* (\filt_{A-1} 
        \fGCt (\grid) / \filt_{A-i-1} \fGCt (\grid));
      \end{equation}
      and
    \item Supposing $i < \np (\grid)$, then $\lambdatp_i (\grid) = 0$ if and 
      only if
      \begin{equation}
        \label{eqn:computation-lambda}
        [\xp (\grid)] = 0 \in \mathrm{H}_* (\filt_{A+i-1} \fGCt (\grid) / 
        \filt_{A-1} \fGCt (\grid)).
      \end{equation}
  \end{itemize}
  (Note that $\widetilde{\bdy}$ is the total differential of $\fGCt (\grid)$.)
  Analogous statements hold for $\nm (\grid)$ and $\lambdatm_i (\grid)$.
\end{proposition}

\begin{remark}
  \label{rmk:diff-from-delta-k}
  \cite[Proposition~1]{NOT08} and \cite[Corollary~1.5]{GJ19} mention homomorphisms 
  $\widehat{\delta}_i \colon \HFLh_d (S^3, L, s) \to \HFLh_d (S^3, L, s-i)$.  
  When $i = 1$, we have that \eqref{eqn:computation-npm} is satisfied if and 
  only if $\widehat{\delta}_i ([\xp (\grid)]) = 0$, since
  \[
    \mathrm{H}_* (\filt_{A-1} \fGCt (\grid) / \filt_{A-2} \fGCt (\grid)) \isom 
    \GHt_{A-1} (\grid).
  \]
  However, for $i > 1$, only one of the two implications holds in general.
\end{remark}

Given \fullref{prop:intro-computation}, the invariants $\npm$ and 
$\lambdatpm_i$ can be computed using a modified version of the ``zigzag'' 
complex in \cite[Section~4]{NOT08}.
Note that, since the Alexander filtration on $\fGCt (\grid)$ is bounded, the spectral sequence collapses in finitely many pages, allowing one to determine when $n^\pm (\grid) = \infty$.
The authors
\cite{JPSWW22:program} have implemented the algorithm described in this subsection in Python.

\subsection{Effectiveness}
\label{ssec:intro-effectiveness}

In \cite{BLW22}, $\lambdatpm$ are shown to be effective in obstructing 
decomposable Lagrangian cobordisms, meaning that they provide obstructions 
beyond the classical invariants.  Precisely, for every $g \geq 0$, there exist 
Legendrian knots $\Legm$ and $\Legp$, such that
\begin{itemize}
  \item $\tb (\Legp) - \tb (\Legm) = 2 g$;
  \item $\rot (\Legp) = \rot (\Legm)$;
  \item There exists a smooth cobordism of genus $g$ between $\Legm$ and 
    $\Legp$; but
  \item $\lambdatp (\Legp) = 0$ and $\lambdatm (\Legm) \neq 0$, implying that 
    there does not exist a decomposable Lagrangian cobordism from $\Legm$ to 
    $\Legp$.
\end{itemize}

Since $\lambdatpm = \lambdatpm_1$, the invariants $\lambdatpm_i$ are indeed 
effective. In view of \fullref{thm:BLW22} and \fullref{rmk:antecedents}, it 
would be interesting to ask:
\begin{itemize}
  \item For $g > 0$, do there exist examples where $\npm$, or $\lambdatpm_i$ 
    with $i \geq 2$, obstruct decomposable Lagrangian cobordisms that 
    $\lambdatpm$ could not?
  \item For $g = 0$, do there exist examples where $\npm$ with $\np (\Legp) 
    \geq 2$, or $\lambdatpm_i$ with $i \geq 2$, obstruct decomposable 
    Lagrangian cobordisms that $\lambdatpm$ could not?
\end{itemize}

While it seems likely that the answer to both questions is in the affirmative,  
the authors have not yet been able to answer the first question. Below, we 
provide an example that answers the second question in the affirmative.

\begin{example}
  \label{eg:m-10-140}
  The pretzel knot $P (-4, -3, 3) = m (10_{140})$ has three undestabilizable 
  Legendrian representatives $\Leg_1$, $-\Leg_1$, and $\Leg_2$, as in the 
  Legendrian knot atlas by Chongchitmate and Ng \cite{CN13atlas}.\footnote{In 
    \cite{CN13atlas}, they are labeled $L_1$, $-L_1$, and $L_2$; we continue to 
    use $\Leg$ for Legendrian links, for consistence. Note also that our 
    $-\Leg_1$ corresponds to $\Leg_1$ in \cite{GJ19}.} Using 
  \cite{JPSWW22:program}, we have computed:
  \begin{align*}
    \lambdatp_1 (\Leg_1) &\neq 0, \qquad
    &\np (\Leg_1) &= 1, \qquad
    &\lambdatm_1 (\Leg_1) &\neq 0, \qquad
    &\nm (\Leg_1) &= \infty;\\
    \lambdatp_1 (-\Leg_1) &\neq 0, \qquad
    &\np (-\Leg_1) &= \infty, \qquad
    &\lambdatm_1 (-\Leg_1) &\neq 0, \qquad
    &\nm (-\Leg_1) &= 1;\\
    \lambdatp_1 (\Leg_2) &\neq 0, \qquad
    &\np (\Leg_2) &= 1, \qquad
    &\lambdatm_1 (\Leg_2) &\neq 0, \qquad
    &\nm (\Leg_2) &= 1.
  \end{align*}
  Golla and Juh\'asz \cite[Proposition~1.6]{GJ19} show that there is no 
  decomposable (in fact, regular\footnote{An exact Lagrangian cobordism $\Lag$ 
    is \emph{regular} if the Liouville vector field is tangent to $\Lag$.}) 
  Lagrangian concordance $\Lag \colon \Leg_2 \to -\Leg_1$, using the fact that 
  $\np (-\Leg_1) > \np (\Leg_2) = 1$.
  This is based on a computation of $\widehat{\delta}_1 (\lambdatp (-\Leg_1)) = 
  0$ and $\widehat{\delta}_1 (\lambdatp (\Leg_2)) \neq 0$ in \cite{NOT08} using 
  the computer program \cite{NOT07:program}.\footnote{Note that $L_1$ and $L_2$ 
    in \cite{NOT08} correspond to $-\Leg_1$ and $\Leg_2$ respectively. There is 
    also a newer program \cite{MQRVW19:program} that provides bug fixes and 
    improvements in computational speed to \cite{NOT07:program}.}
  While not directly stated in \cite{GJ19}, their result on $\np$ also implies 
  that there is no decomposable Lagrangian cobordism $\Lag \colon \Leg_1 \to 
  -\Leg_1$.

  The above is recovered by \fullref{thm:cob}. In addition, by considering 
  $\nm$, \fullref{thm:cob} also implies that there is no decomposable 
  Lagrangian cobordism $\Lag \colon -\Leg_1 \to \Leg_1$ or $\Lag \colon \Leg_2 
  \to \Leg_1$, which was previously unknown.

  Finally, note that none of the obstructions above can be obtained by 
  considering only linearized Legendrian contact homology (LCH), since 
  $\Leg_1$, $-\Leg_1$, and $\Leg_2$ all have the same linearized LCH. This 
  eliminates the most tractable approach to using the Chekanov--Eliashberg DGA 
  from symplectic field theory to obtain an obstruction.
\end{example}

\subsection{Organization}

In \fullref{sec:prelim}, we provide the necessary background on Legendrian 
links, Lagrangian cobordisms, link Floer homology, the GRID invariants, and 
filtered chain complexes and their associated spectral sequences. Next, in 
\fullref{sec:filt-def}, we define the spectral invariants in detail, and 
provide a detailed proof that they are preserved by grid commutation and 
(de)stabilization, which implies that they are Legendrian invariants. In 
\fullref{sec:obstr}, we then show that the spectral invariants are preserved 
under (the reverses of) pinches and births, proving \fullref{thm:cob}.  
Finally, we prove \fullref{prop:intro-computation} in 
\fullref{sec:computation}, establishing an algorithm to compute the spectral 
invariants directly from the filtered chain complex.

\subsection*{Acknowledgments}

This work is the result of the 2022 Summer Hybrid Undergraduate Research (SHUR) 
program at Dartmouth College, and the authors thank Dartmouth for the support.  
IP was partially supported by NSF CAREER Grant DMS-2145090. MW was partially 
supported by NSF Grant DMS-2238131 (previously DMS-2039688). The SHUR program 
was also partially supported by these NSF grants. MW was also partially 
supported by NSERC Discovery Grant RGPIN-2023-05123. Part of the research was 
conducted while MW was at Louisiana State University, and he thanks LSU for the 
support.

\section{Preliminaries}\label{sec:prelim}

\subsection{Legendrian knots and Lagrangian cobordisms}

In this section, we review the basics of Legendrian knots and Lagrangian cobordisms.

Recall that a smooth link $\Leg\in\R^3$ is called \emph{Legendrian} if it is everywhere tangent to the standard contact structure on $\R^3$,
\[
  \xi_{\std}=\ker(\alpha_{\std}),\quad \alpha_{\std}=dz-y\,dx.
\]
Two Legendrian links are Legendrian isotopic if they are isotopic through a family of Legendrian links.

A Legendrian link can be represented by its \emph{front diagram}, or \emph{front projection}, the projection of the link onto the $xz$-plane. In a front diagram, strand crossing information is encoded by the slopes of the strands: strands with lower slope pass over strands with higher slope. See \fullref{fig:4_1} for an example. 

\begin{figure}[ht]
	\labellist
    	\pinlabel {\small{$x$}} at 38 4
   	\pinlabel {\small{$z$}} at 4 38
    	\endlabellist
      	\includegraphics[scale=1]{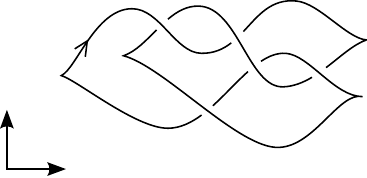}
  	\caption{An example of a front projection.}
    	\label{fig:4_1}
\end{figure}

Two Legendrian front diagrams represent Legendrian-isotopic links if the 
diagrams can be related by a sequence of Legendrian planar isotopies (isotopies 
that preserve left and right cusps and do not introduce vertical tangencies) and Legendrian Reidemeister moves.  
Legendrian Reidemeister moves are the first three diagrams in  
\fullref{fig:leg-moves} and their mirror reflections.

\begin{figure}[ht]
      \includegraphics[scale=1]{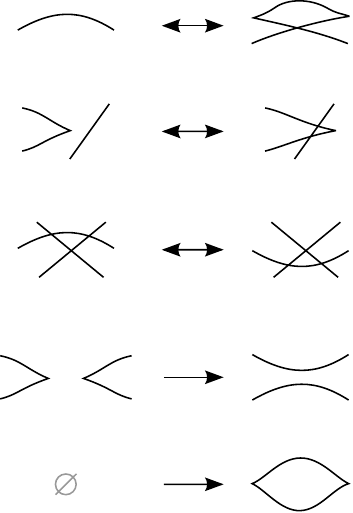}
  \caption{The moves on front projections that correspond to elementary cobordisms. The first three moves are Legendrian Reidemeister moves; they preserve the Legendrian knot type and correspond to Lagrangian cylinders. The fourth move is called a pinch, and the fifth move is called a birth. The vertical and horizontal reflections of the moves are also allowed.}
    \label{fig:leg-moves}
\end{figure}

The two classical Legendrian link invariants are the Thurston-Bennequin number 
$\tb(\Leg)$ and the rotation number $\rot(\Leg)$. These can be computed from an 
oriented front diagram $D$ via the relations
\[
  \tb(\Leg)=\mathrm{wr}(D) - \frac{1}{2} (c_+(D)+c_-(D)), \qquad 
  \rot(\Leg)=\frac{1}{2} (c_-(D)-c_+(D)),
\]
where $\mathrm{wr}(D)$ is the writhe of the diagram, and $c_-(D)$ and $c_+(D)$ 
are the number of downward and upward cusps, respectively.

Two important operations on Legendrian links, which change the Legendrian isotopy class, are positive and negative \emph{Legendrian stabilizations}; see \fullref{fig:Leg-stab}. A link which is not the positive or negative stabilization of another is called \emph{undestabilizable}.
\begin{figure}[ht]
	\labellist
    	\pinlabel $+$ at 78 72
   	\pinlabel $-$ at 78 25
    	\endlabellist
      	\includegraphics[scale=1]{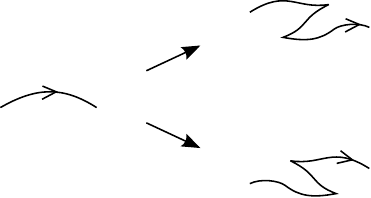}
  	\caption{Positive and negative stabilization of a Legendrian link.}
    	\label{fig:Leg-stab}
\end{figure}

The \emph{symplectization} of $(\R^3, \xistd)$ is the symplectic $4$-manifold
\[
  (\R_t \cross \R^3, d (e^t \alphastd)).
\]
A \emph{Lagrangian cobordism} from $\Legm \subset (\R^3, \xistd)$ to $\Legp 
\subset (\R^3, \xistd)$ is an oriented, embedded surface $L \subset \R_t \cross 
\R^3$ such that
\begin{itemize}
  \item $L$ is Lagrangian, i.e.\ $d (e^t \alphastd) \rvert_L \equiv 0$;
  \item $L$ has cylindrical ends, i.e.\ for some $T > 0$,
    \begin{align*}
      L \cap ((-\infty, -T) \cross \R^3) &= (-\infty, -T) \cross \Legm,\\
      L \cap ((T, \infty) \cross \R^3) &= (T, \infty) \cross \Legp,
    \end{align*}
    and $L \cap ([-T, T] \cross \R^3)$ is compact.
\end{itemize}
A Lagrangian cobordism is \emph{exact} if there exists a function $f \colon L 
\to \R$ that is constant (and not just locally constant) on each of the two 
cylindrical ends, and satisfies
\[
  (e^t \alphastd) \rvert_L = df.
\]
A \emph{Lagrangian concordance} is a Lagrangian cobordism of genus zero, which 
is automatically exact.

If there exists a Lagrangian cobordism $L$ from $\Leg_-$ to $\Leg_+$, 
Chantraine \cite{Cha10:Leg} proves that the classical invariants of the two 
links are related by
\[tb(\Leg_+)-tb(\Leg_-)=-\chi(L) \qquad \text{and} \qquad r(\Leg_+)-r(\Leg_-)=0,\]
where $\chi(L)$ is the Euler characteristic of $L$. This immediately implies 
that Lagrangian cobordism are not an equivalence relation. In fact, even 
Lagrangian concordance is not an equivalence relation \cite{Cha15:NotSym}.

One important subclass of exact Lagrangian cobordisms is the class of 
decomposable Lagrangian cobordisms. Precisely, refer to 
\fullref{fig:leg-moves}: If $\Legm$ and $\Legp$ are Legendrian links such that
\begin{itemize}
  \item $\Legm$ and $\Legp$ are Legendrian isotopic, as in the first three 
    diagrams;
  \item $\Legp$ is obtained from $\Legm$ by a pinch move, as in the fourth 
    diagram;\footnote{Note that, despite the terminology, it is in fact $\Legm$ that looks like it is obtained from $\Legp$ by a pinch.} or
  \item $\Legp$ is obtained from $\Legm$ by a Legendrian birth, i.e.\ $\Legp$ 
    is the disjoint union of $\Legm$ with an unlinked component that is the 
    undestabilizable Legendrian unknot, as in the fifth diagram;
\end{itemize}
then there exists an \emph{elementary} exact Lagrangian cobordism $L \colon 
\Legm \to \Legp$, by work of Bourgeois, Sabloff, Traynor \cite{BST15}, 
Chantraine \cite{Cha10:Leg}, Dimitroglou Rizell \cite{Dim16}, and Ekholm, 
Honda, and K\'alm\'an \cite{EHK16}. Note that, topologically, elementary exact 
Lagrangian cobordisms are annuli, saddles, and cups, respectively. A Lagrangian 
cobordism is \emph{decomposable} if it is isotopic through exact Lagrangian 
cobordisms to a composition of elementary exact Lagrangian cobordisms. In the 
smooth category, every link cobordism is decomposable into elementary 
cobordisms; whether every exact Lagrangian cobordism is decomposable remains a 
major open question.

\subsection{Knot Floer homology and the GRID invariants}
\label{ssec:prelim-grid}

In this section, we review some basics of grid homology, following the 
conventions in \cite{OSS15}.

A \emph{grid diagram} (or simply a \emph{grid}) $\G$ is an $m\times m$ grid on the plane, along with two sets of markers
\[\OO = \{O_1, \ldots, O_m\}, \qquad \X = \{X_1, \ldots, X_m\},\]
such that there is exactly one $O$ and exactly one $X$ in each row, as well as in each column, and no square of the grid contains more than one marking. The number $m$ is called the \emph{grid number} of $\G$.

A grid diagram $\G$ specifies a link $L\subset \R^3$ as follows. Draw oriented segments connecting $X$'s to $O$'s in each column, and $O$'s to $X$'s in each row, and require that vertical segments cross above horizontal ones. We say \emph{$\G$ is a grid diagram for $L$}. Conversely, every link $L$ in $\R^3$ can be represented by a grid diagram.
By a theorem of Cromwell \cite{Cro95}, two grid diagrams represent the same link if and only if they are related by a sequence of commutations, in which two adjacent rows or columns are switched if the corresponding segments in them connecting the $X$'s and the $O$'s are either nested or disjoint, stabilizations, in which a $1 \times 1$ square with an $O$ (resp.\ $X$) marker is replaced by a $2 \times 2$ square with two diagonal $O$ markers and an $X$ marker (resp.\ two diagonal $X$ markers and an $O$ marker), creating a new row and a new column, and destabilizations, the inverse operations. Following \cite{OSS15}, we classify (de)stabilizations by the marker type and the location of the empty cell in the $2 \times 2$ square; for example, a stabilization of type \textit{X:SE} results in a $2 \times 2$ square with an empty southeastern cell, an  $O$ in the northwestern cell, and $X$'s in the northeastern and southwestern cells.

To a grid diagram $\G$, we associate a graded, filtered chain complex 
$\fGCt(\G)$ over $\Ftwo = \Z/2$ whose filtered chain homotopy type is an 
invariant of the isotopy type of $L$. Before we do this, we introduce a bit 
more notation. First, we will think of a grid diagram as a diagram on a torus, 
by identifying the left and right edges, as well as the top and bottom edges of 
the grid. The horizontal arcs of the grid result in a set of circles $\alphas = 
\{\alpha_1, \ldots ,\alpha_m\}$, indexed from bottom to top,  and the vertical 
ones result in a set of circles $\betas = \{\beta_1, \ldots, \beta_m\}$, 
indexed from left to right.\footnote{In \cite{BLW22}, the $\alpha$-circles are 
  the vertical ones, and $\beta$-circles the horizontal ones. Here, we instead 
  follow the convention of \cite{OSS15}. However, in later sections, we opt to 
  follow \cite{BLW22} in drawing multidiagrams (e.g.\ \fullref{fig:comm} and 
  \fullref{fig:combo}) with multiple horizontal curves; as a result, they have 
  multiple $\alpha$-curves rather than multiple $\beta$-curves.} 

As an $\Ftwo$-module, $\fGCt(\G)$ is generated by \emph{grid states}, i.e.\ 
bijections between horizontal and vertical circles. Geometrically, a grid state 
is an $m$-tuple of points $\x = \{x_1, \ldots, x_m\}$ with one point on each 
horizontal circle and one on each vertical circle. The set of grid states for a 
grid diagram $\G$ is denoted by $S(\G)$. 

Before we define the differential, we first define a partial ordering of points in $\mathbb{R}^2$ given by $(x_1,y_1)<(x_2,y_2)$ if $x_1<x_2$ and $y_1<y_2$. For any two sets $P,Q \subset \mathbb{R}^2$, define \[\mathcal{I}(P,Q) = \#\setc{(p,q) \in P\times Q}{p<q}.\] Next, we define the symmetrized function, \[\J(P,Q)=\frac{\I(P,Q)+\I(Q,P)}{2}.\] 

For any state $\x \in S(\G)$, we can define it uniquely as integer points in $[0,m) \times [0,m) \subset \mathbb{R}^2$. Then, representing $\X$ and 
$\OO$ as half-integer points in the same subset, we can define the 
Maslov and Alexander functions $M (\x)$ and $A(\x)$ as follows:
\begin{align*}
    M (\x) = M_{\OO}(\x) & =\J(\x,\x)-2\J(\x,\OO)+\J(\OO,\OO)+1, \\
    M_{\X}(\x) & =\J(\x,\x)-2\J(\x,\X)+\J(\X,\X)+1, \\
    A(\x)&=\frac{1}{2}\Big(M_\OO(\x)-M_\X(\x)\Big)-\frac{m-l}{2},
\end{align*}
where $l$ is the number of link components in the Legendrian link corresponding 
to the grid $\G$.

Given two grid states $\x,\y \in \SG$, let $\Rect(\x, \y)$ denote the space of rectangles embedded in the torus with the following properties. First, $\Rect (\x, \y)$ is empty if $\x$ and $\y$ do not agree at exactly $m - 2$ points. An element $r \in \Rect (\x, \y)$ is an embedded rectangle with right angles, such that:
\begin{itemize}
\item $\bdy r$ lies on the union of horizontal and vertical circles;
\item The vertices of $r$ are exactly the points in $\x \triangle \y$, where $\triangle$ denotes the symmetric difference; and
\item $\bdy (\bdy r \intersect \betas)= \x-\y$, in the orientation induced by $r$.
\end{itemize}
Given $r\in\Rect(\x, \y)$, we say that \emph{$r$ goes from $\x$ to $\y$.} Observe that $\Rect(\x, \y)$  consists of either zero or two rectangles. We say a rectangle $r\in \Rect(\x, \y)$ is \emph{empty} if $\x\cap \Int(r) = \y\cap \Int(r) = \emptyset$. We denote the set of empty rectangles from $\x$ to $\y$ by $\eRect(\x, \y)$.

For any two states $\x,\y$ with a rectangle $r\in \Rect(\x,\y)$, one could 
compute that
\begin{align}
  M (\x) - M (\y) &= 1 - 2 \#(r \cap \OO)+ 2 \# (\Int(r) \cap \x), \label{eq:maslov-gr-rel}\\
  A (\x) - A (\y) &= \# (r \cap \XX) - \# (r \cap \OO). \label{eq:alex-filt-rel}
\end{align}

The differential on $\fGCt(\G)$ is defined on generators by 
\[\bdt(\x) = \sum_{\y\in \SG}\sum_{\substack{r\in \eRect(\x, \y)\\ r\cap \OO = \emptyset}} \y.\]
Using \eqref{eq:maslov-gr-rel} and \eqref{eq:alex-filt-rel}, we examine how the 
differential interacts with the Maslov and Alexander functions. First, note 
that $M(\x)-M(\y)=1$ for each $\y$ in the summation, which implies that $\bdt$ 
drops the Maslov grading by 1. Second, we see that $A(\x)-A(\y)\geq 0$ for all 
$\y$ in the summation.  Thus, $M$ induces a $\Z$-grading on $\fGCt (\grid)$, 
and $A$ induces a $\Z$-filtration.
We discuss filtered complexes further in \fullref{ssec:filt}.

The homology $\grGH(\G) = H_*(\gr(\fGCt(\G)))$ of the associated graded object is (almost) an invariant of the underlying link, in the following sense. If $\G$ is a grid of size $m$ for an $l$-component link $L$, then we have 
\begin{equation*}
    \GHt(\G) \cong \HFLh(L)\otimes W^{\otimes (m-l)} \, ,
\end{equation*}
where $W$ is a two-dimensional bigraded vector space with one generator in bigrading $(0,0)$ and another in bigrading $(-1,-1)$. Alternatively, a combinatorial proof of the invariance of $\GHt$ that does not appeal to holomorphic Heegaard Floer theory is given by defining combinatorial filtered chain homotopy equivalences between $\fGCt (\grid_1)$ and $\fGCt (\grid_2)$ when $\grid_1$ and $\grid_2$ differ by a commutation or (de)stabilization.

Given a planar grid $\G$, we construct a Legendrian link in the following 
manner. First, create the oriented link specified by  $\G$.  The 
projection of this link onto the grid has corners that can be classified into 
four types: northeast, northwest, southwest, and southeast. First, we smooth 
all of the northwest and southeast corners of the projection, and turn the 
northeast and southwest corners into cusps. Next, we rotate the diagram $45$ 
degrees clockwise. Now, since all the vertical strands cross over the 
horizontal ones and this convention is opposite to the convention for 
Legendrian front projections, we now flip all the crossings.
This gives a Legendrian link $\Leg (\grid)$ whose smooth type is the mirror of the smooth link associated to $\grid$. Similar to the smooth case, every Legendrian link in $(\R^3, \xistd)$ can be represented by a grid diagram. Two grid diagrams represent the same Legendrian link if and only if they are related by a sequence of commutations and (de)stabilizations of type \textit{X:SE} and \textit{X:NW} \cite[Proposition 12.2.6]{OSS15}.

Given a grid diagram $\G$, the generator $\xp (\grid) \in \SG$ is the grid state composed of all the points directly northeast of the $X$'s. Similarly, we define $\xm (\grid) \in \SG$ to be the grid state composed of all points directly southwest of the $X$'s.
Then, for a grid diagram $\G$ of a Legendrian link $\Leg$ of $l$ components, one can compute the gradings of the generators $\xpm$ to be
\begin{equation}
\label{eq:xpm grad}
    \begin{split}
        M(\xpm (\grid))&= tb(\Leg) \mp r(\Leg) + 1 \\
        A(\xpm (\grid))&= \frac{tb(\Leg) \mp r(\Leg) + l}{2}.
    \end{split}
\end{equation}
In \cite{OST08}, both $\xp (\grid)$ and $\xm (\grid)$ are shown to be cycles in the associated graded object $\gr (\fGCt (\grid)) = \GCt (\grid)$; moreover, these cycles are preserved by the filtered chain homotopy equivalences associated to commutations and (de)stabilizations of type \textit{X:SE} and \textit{X:NW}, which shows that their homology classes are invariants of the Legendrian link $\Leg$.

\subsection{Filtered chain complexes}\label{ssec:filt}
The main algebraic structures that we study in this paper are filtered chain complexes, and the spectral sequences they induce.

A \emph{filtration} on a chain complex $C$ is a sequence of subcomplexes $(\filt_i C)_{i \in \Z}$ of $C$ such that $\filt_i C \subset \filt_j C$ whenever $i \le j \in Z$. To be more precise, this is the definition of an ``increasing'' filtration. We will assume that our filtrations are \emph{bounded}, which means that $\filt_s C = 0$ and $\filt_t C = C$ for some $s \le t \in \Z$. A \emph{filtered complex} $\fC = (C, \filt)$ is a complex $C$ equipped with a filtration $\filt$ of $C$. When the filtration is clear from context, we will often omit it.

A \emph{map of filtered complexes} is a chain map $f: C \to C'$ that respects the filtration in the sense that $f(\filt_i C) \subset \filt_i C'$.

Given a filtered complex $\fC$, we may construct the \emph{associated graded complex} $\gr(\fC)$, which is defined to be
\begin{equation*}
    \gr(\fC) = \bigoplus_{p \in \Z} \gr_p(\fC)
\end{equation*}
where
\begin{equation*}
    \gr_p(\fC) = \filt_p C / \filt_{p-1} C \,.
\end{equation*}
A map $f \colon \fC \to \fC'$ of filtered complexes induces a map $\gr(f) \colon \gr(\fC) \to \gr(\fC')$ of associated graded complexes in a natural way.

While the underlying modules of filtered complexes and their associated graded 
complexes are isomorphic (over a field), the same is not true in general when 
one considers their differentials. For example, it is not true in general that 
$H_*(\gr(\fC)) \iso H_*(C)$. Instead, what we can say is that there is a 
spectral sequence from the former group to the latter. A \emph{spectral 
  sequence} is a sequence of chain complexes $(E^r)_{r \in \N \union 
  \{\infty\}}$ and isomorphisms $H_*(E^r) \iso E^{r+1}$ for all $r \in \N 
\union \{\infty\}$. The complex $E^r$ comes equipped with a bigrading $E^r = 
\bigoplus_{p,q} E^r_{p,q}$.

Given a filtered complex $\fC$, there is an induced spectral sequence with 
$E^1_{p,q} \iso H_{p+q}(\gr_p(\fC))$ and $E^\infty_{p,q} \iso 
\gr_p(H_{p+q}(C))$. We will summarize the relevant details, using the 
conventions from \cite[\href{https://stacks.math.columbia.edu/tag/012K}{Section 
  012K}]{Sta22} adjusted for homological (vs.\ cohomological) gradings. The 
filtration $\filt$ on $C$ induces a grading on each page $E^r$ of the spectral 
sequence as a module; this is the grading corresponding to $p$ in $E^r_{p,q}$.  
Similarly, $q$ corresponds to the original homological grading on $C$; we will 
often suppress this grading for simplicity. As modules, we define
\begin{equation*}
    E^r_p (\fC) \iso \frac{Z^r_p (\fC)}{B^r_p (\fC)} \,,
\end{equation*}
where
\begin{align*}
    Z^r_p (\fC) \iso \frac{\filt_p C \cap \diff^{-1}(\filt_{p-r} C) + 
      \filt_{p-1}C}{\filt_{p-1} C} \quad \text{ and }
    \quad
    B^r_p (\fC) \iso \frac{\filt_p C \cap \diff(\filt_{p+r-1} C) + \filt_{p-1} 
      C}{\filt_{p-1} C} \,.
\end{align*}
The differential $d^r: E^r_p \to E^r_{p-r}$ is induced by $\diff$, where 
  $E^r$ is understood as a quotient of subquotients of $C$ as in the definition 
  above.  
The only part of the spectral sequence structure induced by $\filt$ that we 
have not defined is the isomorphism $H_*(E^r) \iso E^{r+1}$; this is mostly 
tedious but straightforward algebra.  For more details, we refer the reader to 
\cite{Wei94} or another book on homological algebra.

Since we are interested in the behavior of particular elements under spectral 
sequences, we want to define what it means to talk about the ``class of $x \in 
C$ on the $r$-th page''. Given any nonzero $x \in C$, let $p$ be the value  for which $x \in \filt_p C \setminus \filt_{p-1} C$ (note that this is 
the difference as sets and not the quotient). We may think of $p$ as the 
``filtration level'' of $x$; such an integer always exists since $\filt$ is 
assumed to be bounded. Then, if $x \in \filt_p C \cap \diff^{-1}(\filt_{p-r} 
C) + \filt_{p-1} C$, we define $[x]^r$ to be the class of $x$ in $E^r_p$.  
Note that, given some $x \in \filt_p C$, it may not be the case that $[x]^r$ 
is defined for all $r$. However, if $[x]^r$ is defined, then $[x]^s$ is also 
defined for all $s \le r$, since $\diff^{-1}(\filt_{p-s} C) \supseteq 
\diff^{-1}(\filt_{p-r} C)$. In fact, $[x]^{r+1} \in E^{r+1}_p$ is defined if and only if $d^r[x]^r = 0$, since $\ker d^r_p = \diff^{-1}(B^r_{p-r}) \cap Z^r_p \iso 
Z^{r+1}_p$. In \fullref{sec:computation}, we will introduce an alternative way 
of thinking about the (non)vanishing of the class of an element in $E^r$ 
that lends itself nicely to certain computations.

In later sections, our strategy to relate elements in the spectral sequences associated to two filtered complexes will be to relate representatives of those elements in the filtered chain complexes. Thus, we need the following lemma, which explains how a filtered chain map induces maps on the spectral sequence:

\begin{lemma}
  \label{lem:filt-map-ss-el}
  Let $\fC = (C, \filt)$ and $\fC' = (C', \filt')$ be two filtered chain 
  complexes, and let $f \colon \fC \to \fC'$ be a filtered chain map. Then for 
  each $r \geq 1$, the map $f$ induces a chain map
  \[
    E^r (f) \colon E^r (\fC) \to E^r (\fC').
  \]
  Furthermore, if $x \in \filt_p C$ is an element with a well-defined class 
  $[x]^r \in E^r (\fC)$, such that $f (x)$ can be (non-uniquely) written as
  \[
    f (x) = y_1 + y_2
  \]
  where $y_1 \in \filt_p' C'$ and $y_2 \in \filt_{p-1}'C'$, then $[y_1]^r$ is 
  well-defined, and
  \[
    E^r (f) ([x]^r) = [y_1]^r.
  \]
\end{lemma}

\begin{proof}
    First, assuming we are given a map $f: \fC \to \fC'$, we will describe the map $E^r(f)$. Since $E^r(\fC)$ and $E^r(\fC')$ are sub-quotients of $\fC$ and $\fC'$ respectively, one can check that $f: \fC \to \fC'$ induces a map $E^r(f): E^r(\fC) \to E^r(\fC')$. On elements, $E^r(f)$ is defined such that
    \begin{equation*}
        E^r(f)([x]^r) = [f(x)]^r \,.
    \end{equation*}
    Note that $f$ is a filtered chain map and therefore commutes with the original differentials, i.e.\ $f\circ \diff = \diff' \circ f$. Therefore, we get that $E^r(f)$ also commutes with the induced differentials, i.e.\ $E^r(f) \circ d^r = (d')^r \circ E^r(f)$. It can be shown further that $E^r(f)$ is also the map induced by $E^{r-1}(f)$ on homology (a property shared by all morphisms of spectral sequences).

    Next, assume $x \in \filt_p C$ is an element with a well-defined class $[x]^r \in E^r(\fC)$, and that
    \begin{equation*}
        f(x) = y_1 + y_2
    \end{equation*}
    for some $y_1 \in \filt_p C'$ and $y_2 \in \filt_{p-1} C'$. Then
    \begin{equation*}
        E^r(f)([x]^r) = [f(x)]^r = [y_1 + y_2]^r = [y_1]^r
    \end{equation*}
    by the definition of $E^r(f)$ and the fact that $y_2 \in \filt_{p-1} C'$ is in the denominator of $Z^r(\fC')$.
\end{proof}

\section{Spectral GRID invariants}\label{sec:filt-def}

\subsection{Definition of the spectral GRID invariants} \label{sec:invariant-definitions}

Now, we have the necessary background to define our invariants.

\begin{definition}\label{def:n-g}
Suppose that $\G$ is a grid diagram, and let $A = A(\xp(\G))$. We define $\np(\G)$ to be the smallest integer $i$ for which $d^i_A [\xp(\G)]^i \neq 0 \in E^i_A$, or $\infty$ if $d^i_A [\xp(\G)]^i = 0$ for all $i \in \Z_{\geq 1}$. We define $\nm(\G)$ analogously, replacing $\xp(\G)$ by $\xm(\G)$.
\end{definition}

\begin{definition}\label{def:lambda-i-g}
Suppose that $\G$ is a grid diagram. For each $1 \leq i \leq \np(\G)$, we define
\[
  \lp_i(\G) = [\xp(\G)]^i \in E^i_{A(\xp(\G))}, \qquad \lm_i(\G) = [\xm(\G)]^i \in E^i_{A(\xm(\G))}.
\]
\end{definition}

For the rest of this section, we will focus on proving \fullref{thm:leg}, which 
states the invariance of
$\npm (\grid)$ and $\lambdat_i (\grid)$ under the choice of $\grid$.
This will allow us to denote them by $\npm (\Leg)$ and $\lambdat_i (\Leg)$.

In \cite{OST08, OSS15}, invariance of $\lambdatpm$ is proven by considering 
the isomorphisms, on the homology level, associated to commutation and 
destabilization, and showing that they carry $[\xpm]$ to $[\xpm]$. For filtered 
invariants, the isomorphisms are shown to be covered by filtered 
quasi-isomorphisms on the chain level. We take a very similar approach, but 
with two differences:
\begin{itemize}
  \item We work directly on the filtered chain level, and show that $\xpm$ is 
    carried by the filtered chain map either to $\xpm$, or to $\xpm + \y$, 
    where $\y$ belongs to a lower filtration; and
  \item We do not require our destabilization maps to be isomorphisms on 
    homology (of the associated graded object); this is in line with the 
    philosophy of viewing these maps as maps of decorated link cobordisms.  
    (See, for example, \cite{Zem19}.) Accordingly, we also separately consider 
    stabilization (and not just destabilization) maps.
\end{itemize}

\subsection{Commutation}\label{sec:commutation}

First, we prove invariance under commutation.

\begin{lemma}\label{lem:comm}
Suppose $\G_1$ and $\G_2$ differ by a commutation move.
Then there exists a filtered chain homomorphism
\[\commmap\colon  \fGCt(\G_2)\to  \fGCt(\G_1) 
\]
such that \[\commmap(\xpm(\G_2)) = \xpm(\G_1) + \y,
\]
where $\y \in \filt_{A (\xpm (\grid_2)) - 1} \fGCt (\grid_1)$.
\end{lemma}

\begin{proof}
This is essentially \cite[Lemma~6.6]{OST08}.
We briefly recall here the definition of the map constructed there, as we will also need it in \fullref{ssec:pinch}. The map is defined by a count of pentagons, as follows. Suppose the commutation is a row commutation, and combine $\G_1$ and $\G_2$ into one diagram as in \fullref{fig:comm}, with $\alpha$ corresponding to $\G_1$ and $\alpha'$ corresponding to $\G_2$.
\begin{figure}[ht]
         \labellist
    \pinlabel $O$ at 76 37
    \pinlabel $X$ at 20 37
    \pinlabel $O$ at 169 37
    \pinlabel $X$ at 130 37
   \pinlabel {\small{$a$}} at 113 48
    \pinlabel \textcolor{Maroon}{$\alpha$} at -10 46
    \pinlabel \textcolor{Orange}{$\alpha'$} at -10 28
    \endlabellist
  \includegraphics[scale=1]{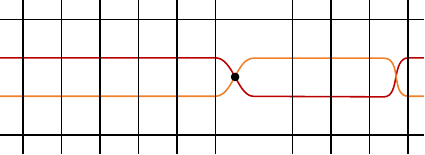}
  \caption{The combined diagram of a row commutation involving $\alpha$ and $\alpha'$.}
    \label{fig:comm}
\end{figure}
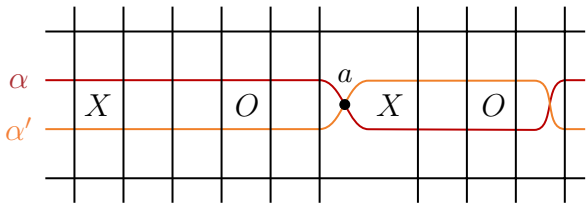
For $\x \in S(\G_2)$ and $\y \in S(\G_1)$, let $\Pent(\x,\y)$ be the space of pentagons in the combined diagram with the following properties. First, $\Pent(\x,\y)$ is empty if $\x$ and $\y$ do not agree at exactly $m-2$ points. An element $p \in \Pent(\x,\y)$ is an embedded pentagon with non-reflex angles whose boundary lies on the horizontal and vertical circles (including $\alpha$ and $\alpha'$) and whose vertices are points in $(\x \triangle \y) \cup \{a\}$, where $\triangle$ denotes the symmetric difference, such that $\bdy (\bdy p \cap \betas) = \x - \y$ in the induced orientation.
Let $\ePent(\x,\y)$ be the subset of those $p \in \Pent(\x,\y)$ such that $\Int(p) \cap \x = \emptyset$,
and $\ePent_{\OO}(\x,\y)$ the subset of those $p \in \Pent(\x,\y)$ such that $\Int (p) \cap \x = p \cap \OO = \emptyset$.
Then, we define a linear map $\commmap$ on $\fGCt(\G_2)$ given by
\[ \commmap(\x) = \sum_{\y\in S(\G_1)} \sum_{p \in \ePent_{\OO}(\x,\y)} \y.\] 

This map is known to respect the Maslov grading and the Alexander filtration by 
\cite[Lemma~3.1]{MOST07}.
In
the proof of \cite[Lemma~6.6]{OST08}, only pentagons that do not contain $X$'s 
are considered,
and it is asserted that $\commmap(\xpm(\G_2)) = \xpm(\G_1)$; indeed, there is 
only one pentagon from $\xpm (\grid_2)$
that does not contain $X$'s (or intersection points in $\x$),
which gives the term 
$\xpm (\grid_1)$. Allowing pentagons that contain $X$'s (but blocking those that contain $O$'s), all other pentagons
from $\xpm (\grid_2)$ contain at least one $X$, which means that the Alexander 
filtration of the target generator $\y \neq \xpm (\grid_1)$ must be lower.
\end{proof}

\subsection{Stabilization}\label{sec:stabilization}

Next, we consider stabilization maps.

\begin{lemma}\label{lem:stab}
  Suppose $\grid_1$ is obtained from $\grid_2$ by a type \textit{X:SE} or 
  \textit{X:NW} stabilization. Then there exists a filtered chain homomorphism
  \[
    \stabmap \colon \fGCt (\grid_2) \to \fGCt (\grid_1)
  \]
  such that
  \[
    \stabmap (\xpm (\grid_2)) = \xpm (\grid_1) + \y,
  \]
  where $\y \in \filt_{A (\xpm (\grid_2)) - 1} \fGCt (\grid_1)$.
\end{lemma}

\begin{proof}
  We focus first on a type \textit{X:SE} stabilization. In this case, we define 
  the filtered chain homomorphism $\stabmap$ to be the map 
  $\stabmap^{\mathit{oL}}$ in \cite[Definition~14.3.6]{OSS15} with all formal 
  variables $V_i$ set to zero. We now recall $\stabmap^{\mathit{oL}}$ with this 
  modification.

  Let $c$ be the intersection point of the two new curves in $\grid_1$, and let $X_1$ and $X_2$ be the new $X$ markings in the southwestern and northeastern cell of the distinguished $2\times 2$ square,  as 
  shown in \fullref{fig:stab}.  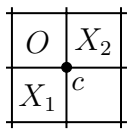
\begin{figure}[ht]
            \vspace{.2cm}
        \labellist
            \pinlabel $O$ at 12 33
            \pinlabel $X_1$ at 12 12
            \pinlabel $X_2$ at 33 33
            \pinlabel {\small{$c$}} at 27 17
        \endlabellist
  \includegraphics[scale=1]{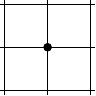}
  \caption{The distingushed $2\times 2$ square of the diagram  $\G_1$ obtained from $\G_2$ by a type \textit{X:SE} stabilization.}
  \label{fig:stab}
\end{figure}
Decompose 
  $\genset{\grid_1}$ into the disjoint union $I (\grid_1) \disjunion N 
  (\grid_1)$, where $I (\grid_1)$ consists of $\x \in \genset{\grid_1}$ such 
  that $c \in \x$, and $N (\grid_1) = \genset{\grid_1} \setminus I 
  (\grid_1)$. Writing $I$ and $N$ for the corresponding submodules, this 
  allows us to decompose
  \(
    \fGCt (\grid_1),
  \)
  as a module, as $I \dirsum N$. Note that, as a chain complex, $\fGCt 
  (\grid_1)$ is not the direct sum of $I$ and $N$, or even a mapping cone of 
  them; the differential consists of maps between $I$ and $N$ in both 
  directions.
  
  In what follows, given $\x, \y \in \SG$, a \emph{domain $p$ from $\x$ 
      to $\y$} is a formal linear combination of the closures of the squares in 
    $\grid$, such that $\bdy (\bdy p \cap \betas) = \x - \y$ in the induced 
    orientation on $\bdy p$. The space of all domains from $\x$ to $\y$ is 
    denoted by $\pi (\x, \y)$. In particular, $\Pent (\x, \y) \subset \pi (\x, 
    \y)$. See, for example, \cite[Definition~4.6.4]{OSS15}.

  Now for $\x \in I (\grid_1)$ and $\y \in \genset{\grid_1}$, a domain $p \in
  \pi (\x, \y)$ is said to be \emph{of type oL} (originally for ``out of the 
  left'') if it is trivial, in which case it has \emph{complexity} $1$, or it 
  satisfies the following conditions:
  \begin{itemize}
    \item All the local multiplicities of $p$ are non-negative;
    \item At each corner in $\x \union \y \setminus \set{c}$, at least three of 
      the four adjacent squares have vanishing local multiplicities;
    \item The domain $p$ has local multiplicity $k-1$ at the southeast square 
      adjacent to $c$, and local multiplicity $k$ at the other three squares 
      adjacent to $c$; and
    \item The generator $\y$ has $2k+1$ coordinates not in $\x$.
  \end{itemize}
  The \emph{complexity} of a non-trivial type \textit{oL} domain is defined to 
  be $2k+1$. See \fullref{fig:stab-domains} for examples of type \textit{oL} 
  domains of complexities $1$, $3$, $5$, and $7$. The set of domains of type \textit{oL} 
  from $\x$ to $\y$ is denoted by $\pi^{\mathit{oL}} (\x, \y)$. In the literature, 
  such domains are often called \emph{snail domains}.
  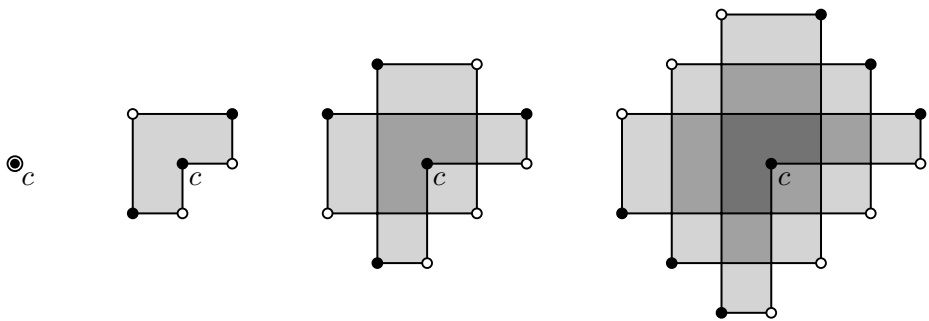
\begin{figure}[ht]
          \labellist
            \pinlabel {\small{$c$}} at 8 53
            \pinlabel {\small{$c$}} at 71 53
            \pinlabel {\small{$c$}} at 163 53
            \pinlabel {\small{$c$}} at 293 53
        \endlabellist
      \includegraphics[scale=1]{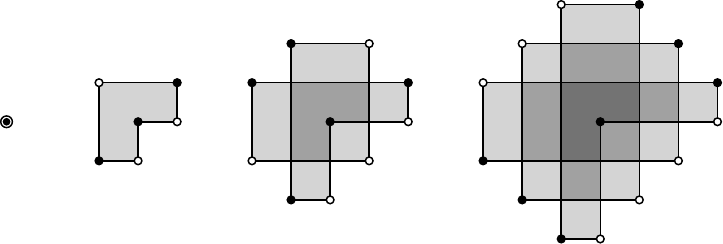}
  \caption{Examples of domains of type \textit{oL}. For each domain, the initial generator is represented by solid circles, and the terminal one by hollow circles.}
    \label{fig:stab-domains}
\end{figure}

  We are ready to define
  \[
    \stabmap = \stabmap^{\mathit{oL}} \colon \fGCt (\grid_2) \to \fGCt 
    (\grid_1)
  \]
  linearly by
  \[
    \stabmap^{\mathit{oL}} (\x) = \sum_{\y \in \genset{\grid_1}} \sum_{p \in 
      \pi^{\textit{oL}} (e' (\x), \y)} \y,
  \]
  where $e' \colon \genset{\grid_2} \to I (\grid_1)$ is the bijection $\x 
  \mapsto \x \union \set{c}$. By (reducing from the \emph{minus} to the 
  \emph{tilde} versions of) \cite[Lemmas~13.3.12 and 14.3.8]{OSS15}, the map 
  $\stabmap^{\mathit{oL}}$ is a chain map that respects the Maslov grading and the Alexander filtration. (Note that, on the level of 
  associated graded objects, $\stabmap^{\mathit{oL}}$ is just the map $e'$, 
  which is the map on the top left of \cite[Figure~5.13]{OSS15}.)

  We now investigate the image of $\xp (\grid_2)$ under the map $\stabmap$.  
  First of all, the trivial domain in $\pi^{\mathit{oL}} (e' (\xp (\grid_2)), 
  e' (\xp (\grid_2)))$ obviously contributes a term. Since clearly $e' (\xp 
  (\grid_2)) = \xp (\grid_1)$, this term is exactly $\xp (\grid_1)$. Now for 
  $\y \neq e' (\xp (\grid_2))$, the domain $p \in \pi^{\mathit{oL}} (e' (\xp 
  (\grid_2)), \y)$ contains at least two $X$'s and one $O$; by an argument 
  analogous to the proof of \cite[Lemma~13.3.12]{OSS15}, we must have $\y \in 
  \filt_{A (\x) - r} \fGCt (\grid_1)$, where $r \geq 3$.
  (Since $\xp (\grid_1)$ occupies the intersection point to the northeast of 
  $X_2$, such a domain $p$ must have multiplicity $1$ in the square containing 
  $X_2$, which means that $p$ must have complexity $3$, i.e.\ is the 
  second-simplest kind of type \textit{oL} domain.)
  This shows that $\xp (\grid_1)$ is the unique generator that appears in the 
  image of $\xp (\grid_2)$ that has the same Alexander filtration level.

  A similar proof applies for $\xm$. Finally, for a type \textit{X:NW} 
  stabilization, the proof above can be adapted, with $\stabmap$ now counting 
  domains that are obtained from the ones above by rotation in the plane by 
  $\pi$.
\end{proof}

\begin{remark}
  \label{rmk:stab-map-oss}
  In \cite[Section~14.3]{OSS15}, the image of $\xpm (\grid_2)$ under $\stabmap$ 
  is not computed; instead, it is observed in \cite[Lemma~14.3.9]{OSS15} that the 
  filtered stabilization map, made up of $\stabmap^{\mathit{oL}}$ and another 
  map $\stabmap^{\mathit{oR}}$, induces a map on the associated graded objects 
  $\GCt$ (or rather, $\GCm$ there) that sends $\xpm$ to $\xpm$. This is 
  sufficient to cover the case $i = 1$ in \fullref{thm:leg}.
\end{remark}

\subsection{Destabilization}
\label{ssec:destab}

We now move on to destabilization.

\begin{lemma}\label{lem:destab}
  Suppose $\grid_1$ is obtained from $\grid_2$ by a type \textit{X:SE} or 
  \textit{X:NW} stabilization. Then there exists a filtered chain homomorphism
  \[
    \destabmap \colon \fGCt (\grid_1) \to \fGCt (\grid_2)
  \]
  such that
  \[
    \destabmap (\xpm (\grid_1)) = \xpm (\grid_2).
  \]
\end{lemma}

\begin{proof}
  This is essentially proved in \cite[Lemma~6.5]{OST08}, but we opt to present 
  a proof here for consistency with the more modern choice of destabilization 
  maps as in \cite{OSS15}.
  
  We again focus on a type \textit{X:SE} destabilization first.  Like 
  $\stabmap$, we also define $\destabmap$ using snail domains.  While a 
  filtered chain map is not spelled out in \cite{OSS15} for a type 
  \textit{X:SE} destabilization, we may draw inspiration from the right half of 
  the commutative diagram in \cite[Figure~5.13]{OSS15} (also alluded to in the 
  proof of \fullref{lem:stab} above) to figure out which snail domains to use.  
  The chain complexes in \cite[Figure~5.13]{OSS15} are those on the level of 
  associated graded objects in our context. (For example, note that for us, 
  there is also an arrow from $I$ to $N$.) Our snail domains should include the 
  maps $I \to L$ and $N \to L$ in that diagram.

  With this insight, we define $\destabmap$ as follows. We continue our 
  notation of $c$, $I (\grid_1)$, $N (\grid_1)$, and $I$ and $N$ as in the 
  proof of \fullref{lem:stab}. For $\x \in \genset{\grid_1}$ and $\y \in I 
  (\grid_1)$, a domain $p \in \pi (\x, \y)$ is said to be \emph{of type iL} 
  (originally for ``into the left'') if it is trivial, in which case it has 
  \emph{complexity} $1$, or it satisfies the following conditions:
  \begin{itemize}
    \item All the local multiplicities of $p$ are non-negative;
    \item At each corner in $\x \union \y \setminus \set{c}$, at least three of 
      the four adjacent squares have vanishing local multiplicities;
    \item The domain $p$ has local multiplicity $k-1$ at the northeast square 
      adjacent to $c$, and local multiplicity $k$ at the other three squares 
      adjacent to $c$; and
    \item The generator $\y$ has $2k+1$ coordinates not in $\x$.
  \end{itemize}
  The \emph{complexity} of a non-trivial type \textit{iL} domain is defined to 
  be $2k+1$. See \fullref{fig:destab-domains} for examples of type \textit{iL} 
  domains of complexities $1$, $3$, $5$, and $7$. The set of domains of type \textit{iL} of 
  complexity $1$ from $\x$ to $\y$ is denoted by $\pi^{\mathit{iL}}_1 (\x, \y)$, 
  while the set of domains of type \textit{iL} of complexity greater than $1$ 
  from $\x$ to $\y$ is denoted by $\pi^{\mathit{iL}}_{>1} (\x, \y)$.
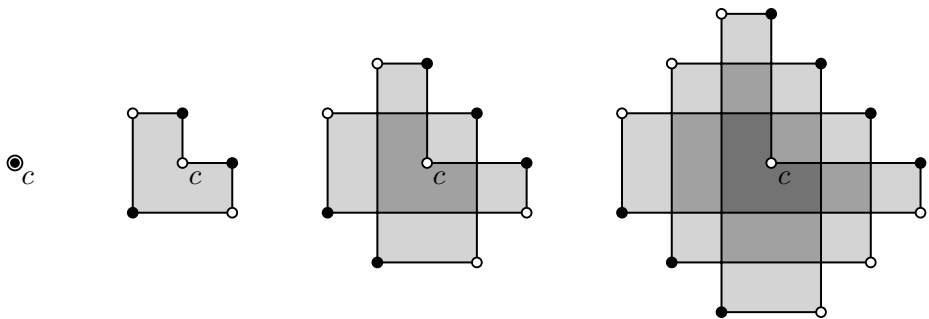
\begin{figure}[ht]
          \labellist
            \pinlabel {\small{$c$}} at 8 53
            \pinlabel {\small{$c$}} at 71 53
            \pinlabel {\small{$c$}} at 163 53
            \pinlabel {\small{$c$}} at 293 53
        \endlabellist
      \includegraphics[scale=1]{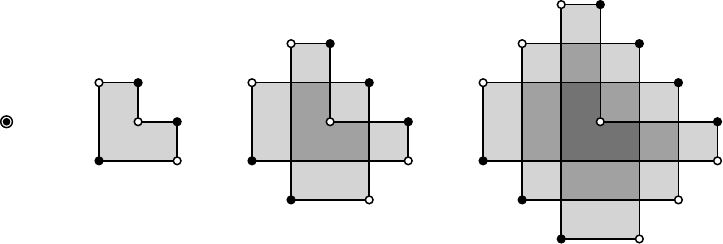}
  \caption{Examples of domains of type \textit{iL}. For each domain, the initial generator is represented by solid circles, and the terminal one by hollow circles.}
    \label{fig:destab-domains}
\end{figure}

  We now define
  \[
    \destabmap = \destabmap^{\mathit{iL}} \colon \fGCt (\grid_1) \to \fGCt 
    (\grid_2)
  \]
  linearly by
  \begin{equation}
    \label{eq:destabmap}
    \destabmap^{\mathit{iL}} (\x) =
    \begin{cases}
      \displaystyle
      \sum_{\y \in I (\grid_1)} \sum_{p \in \pi^{\mathit{iL}}_1 (\x, 
          \y)} e (\y) & \text{if } \x \in I (\grid_1),\footnotemark\\
      \displaystyle
      \sum_{\y \in I (\grid_1)} \sum_{p \in \pi^{\mathit{iL}}_{>1} 
          (\x, \y)} e (\y) & \text{if } \x \in N (\grid_1),
    \end{cases}
  \end{equation}
  \footnotetext{Note that the double sum in this first line is in fact simply $e (\x)$.}
  where $e \colon I (\grid_1) \to S (\grid_2)$ is the bijection $\x \union 
  \set{c} \mapsto \x$. In other words, if we separate the two cases in 
  \eqref{eq:destabmap} into maps
  \[
    \destabmap^{\mathit{iL}}_1 \colon I \to \fGCt (\grid_2), \qquad
    \destabmap^{\mathit{iL}}_{>1} \colon N \to \fGCt (\grid_2),
  \]
  then $\destabmap \colon \fGCt (\grid_1) \to \fGCt (\grid_2)$ is given by the 
  horizontal maps in
  \[
    \xymatrix{
      I \ar@{<->}[d]_{\widetilde{\bdy}} \ar[rr]^{\destabmap^{\mathit{iL}}_1}
      & & \fGCt (\grid_2)\\
      N \ar[urr]_{\destabmap^{\mathit{iL}}_{>1}}
    }.
  \]
  See \cite[(13.6)]{OSS15} for comparison. (The destabilization there is of 
  type \textit{X:SW}.)

  By a case analysis similar to \cite[Lemmas~13.3.12 and 13.3.13]{OSS15}, one 
  could see that $\destabmap$ is a chain map that respects the Maslov 
  grading and the Alexander filtration. For brevity, we omit the details, which are considerably tedious.  
  Alternatively, one could also appeal to the holomorphic Heegaard Floer 
  theory, e.g.\ via \cite[Proposition~3.10 and Lemma~3.11]{MOT20}. (For the 
  astute reader, the map $\destabmap^{\mathit{iL}}_{>1}$ is needed to ensure 
  that we have a chain map.)

  We now investigate the image of $\xpm (\grid_1)$ under the map $\destabmap$.  
  Since $\xpm (\grid_1) \in I (\grid_1)$, we immediately get that
  \[
    \destabmap (\xpm (\grid_1)) = e (\xpm (\grid_1)) = \xpm (\grid_2),
  \]
  which is what we wanted to prove. The case of type \textit{X:NW} 
  destabilizations is handled again by rotating domains by $\pi$.
\end{proof}

\begin{proof}[Proof of \fullref{thm:leg}]
  The result is obtained by applying \fullref{lem:filt-map-ss-el} to the filtered chain homomorphisms in \fullref{lem:comm}, \fullref{lem:stab}, and \fullref{lem:destab}.
\end{proof}

\section{Obstructions to decomposable Lagrangian cobordisms} \label{sec:obstr}

With the goal of proving \fullref{thm:cob}, we will define filtered chain maps corresponding to (the reverses of) pinch and birth moves such that the induced maps on spectral sequences preserve $\lp_i$ and $\lm_i$.

\subsection{Pinches}\label{ssec:pinch}
Suppose $\Leg_+$ and $\Leg_-$ are Legendrian links such that $\Leg_+$ is obtained from $\Leg_-$ by a pinch move. Then there exist diagrams $\G_+$ and $\G_-$ for $\Leg_+$ and $\Leg_-$, respectively, which only differ in the placement of a single pair of $X$'s or $O$'s in adjacent rows, as shown in \fullref{fig:sep}. If the two markings at which the diagrams differ are $X$'s, we will say that $\G_+$ is obtained from $\G_-$ by an $X$ swap; if the markings are $O$'s, we will say that $\G_+$ is obtained from $\G_-$ by an $O$ swap.

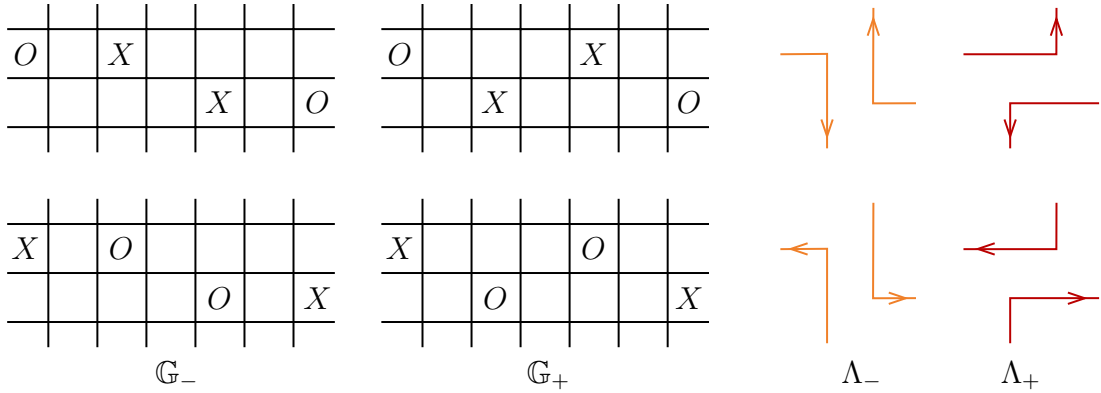
\begin{figure}[ht]
        \labellist
    \pinlabel $O$ at 7 111
    \pinlabel $X$ at 43 111
    \pinlabel $X$ at 80 93
    \pinlabel $O$ at 117 93
    \pinlabel $X$ at 7 38
    \pinlabel $O$ at 43 38
    \pinlabel $O$ at 80 19
    \pinlabel $X$ at 117 19
    \pinlabel $O$ at 148 111
    \pinlabel $X$ at 221 111
    \pinlabel $X$ at 184 93
    \pinlabel $O$ at 257 93
    \pinlabel $X$ at 148 38
    \pinlabel $O$ at 221 38
    \pinlabel $O$ at 184 19
    \pinlabel $X$ at 257 19
    \pinlabel $\G_-$ at 64 -10
    \pinlabel $\G_+$ at 204 -10
    \pinlabel $\Lambda_-$ at 322 -10
    \pinlabel $\Lambda_+$ at 382 -10
    \endlabellist
     \includegraphics[scale=1]{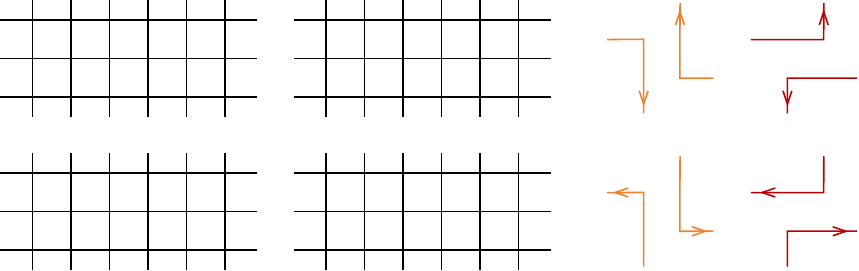}
    \vspace{0.5cm}
  \caption{The grid diagrams $\G_\pm$ corresponding to the two types of pinch moves. The top corresponds to an $X$ swap, and the bottom to an $O$ swap.}
    \label{fig:sep}
    \end{figure}

Through a series of local stabilization and commutation moves, we can ensure 
that the swap occurs between two markings that are separated by at least two 
vertical lines. This can be achieved, for example, by performing a 
stabilization of type \textit{X:SE} on any $X$ marking to the right of all four 
markings in the two adjacent rows where the swap is to be done, and then commuting the column containing the new $O$ to the left 
until it occupies the column between the two markings to be swapped.  
Alternatively, one could perform a stabilization of type \textit{X:NW} on any 
$X$ marking to the left of all four markings in the swap rows, and commute the column with the new 
$O$ to the right.

We combine the two diagrams $\G_+$ and $\G_-$ into a single diagram, as in \fullref{fig:combo}, which we call the \emph{combined diagram}. On the combined diagram, we can see each of $\G_+$ and $\G_-$ by using the same markings, but varying the placement of one horizontal circle ($\alpha$ for $\G_+$ and $\alpha'$ for $\G_-$, as seen in \fullref{fig:combo}).
Of particular 
interest are the intersection points of $\alpha$ and $\alpha'$ labeled  
$a$ and $b$ in \fullref{fig:combo}; these points will be used to define maps $\pinchmap_{X}$ and $\pinchmap_{O}$. 

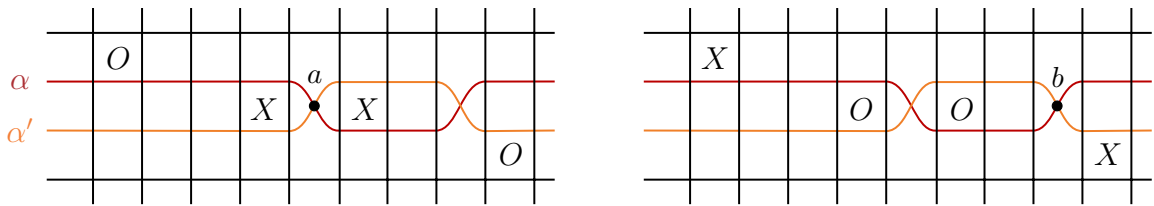
\begin{figure}[ht]
        \labellist
    \pinlabel $O$ at 27 55
    \pinlabel $X$ at 82 36
    \pinlabel $X$ at 120 36
    \pinlabel $O$ at 175 19
    \pinlabel {\small{$a$}} at 101 48
    \pinlabel $X$ at 252 55
    \pinlabel $O$ at 307 36
    \pinlabel $O$ at 345 36
    \pinlabel $X$ at 400 19
    \pinlabel {\small{$b$}} at 381 48
        \pinlabel \textcolor{Maroon}{$\alpha$} at -10 46
    \pinlabel \textcolor{Orange}{$\alpha'$} at -10 28
    \endlabellist
  \includegraphics[scale=1]{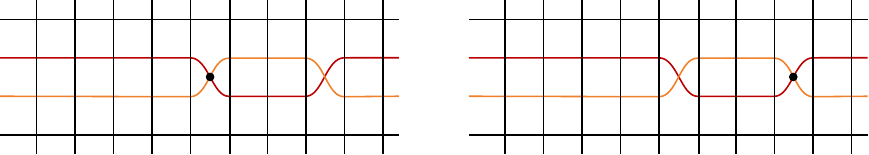}
  \caption{The combined grid diagrams corresponding to the two types of pinch moves. Left: An $X$ swap. Right: An $O$ swap.}
\label{fig:combo}
\end{figure}

\newpage

\subsubsection{\texorpdfstring{$X$ swap}{X swap}}
\begin{lemma}\label{lem:xswap}
Suppose $\G_+$ is obtained from $\G_-$ by an $X$ swap. Then there exists a filtered chain homomorphism
\begin{equation*}
\pinchmap_{X} \colon \fGCt(\G_+) \to \fGCt(\G_-) \left\llbracket 1, \frac{\abs{\Legp} - \abs{\Legm} + 1}{2} \right\rrbracket
\end{equation*}
such that
\[ \pinchmap_X(\xpm(\G_+))=\xpm(\G_-) + \y, \]
where $\y \in \filt_{A (\xpm (\grid_-)) - 1} \fGCt (\grid_-)$.
\end{lemma}

\begin{proof}
 Given $\x\in S(\G_+)$ and $\y\in S(\G_-)$, define $\ePent(\x,\y)$ and $\ePent_{\OO}(\x,\y)$ as in \fullref{sec:commutation}, and let
\[ \pinchmap_X \colon \fGCt(\G_+)\to \fGCt(\G_-) \]
be the linear map that counts these pentagons, i.e.\
\[ \pinchmap_X (\x)= \sum \limits_{\y \in S(\G_-)} \sum \limits_{p\in \ePent_\OO(\x,\y)} \y. \]

First, we prove that $\pinchmap_X$ is a chain map. The proof is analogous to \cite[Lemma 3.1]{MOST07}.  Every domain that is a juxtaposition of a pentagon $p$ and rectangle $r$ decomposes in exactly two ways, and thus contributes an even number of times to the count of $\dmi \circ \, \pinchmap_X(\x) + \pinchmap_X \,\circ \,\dpl(\x)$. In other words, the coefficient of any $\y\in S(\G_-)$ in $\dmi \circ \, \pinchmap_X(\x) + \pinchmap_X \,\circ \,\dpl(\x)$ is zero.

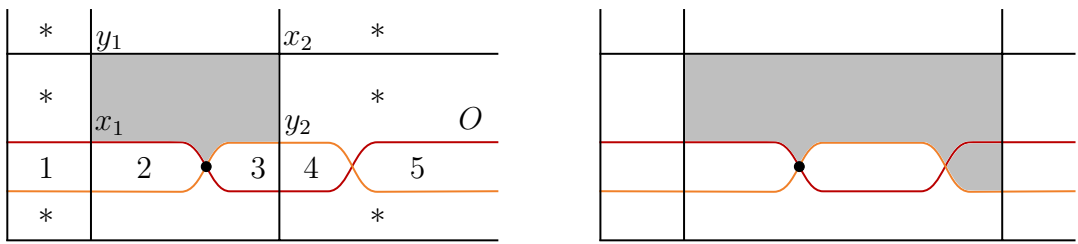
\begin{figure}[ht]
  \vspace{.2cm}
  \labellist
  \pinlabel $x_1$ at 39 44
  \pinlabel $x_2$ at 110 76
  \pinlabel $y_1$ at 39 76
  \pinlabel $y_2$ at 110 44
  \pinlabel $*$ at 15 80
  \pinlabel $*$ at 15 55
  \pinlabel $1$ at 15 28
  \pinlabel $2$ at 52 28
  \pinlabel $3$ at 95 28
  \pinlabel $4$ at 115 28
  \pinlabel $5$ at 155 28
  \pinlabel $O$ at 175 47
  \pinlabel $*$ at 15 10
  \pinlabel $*$ at 140 80
  \pinlabel $*$ at 140 55
  \pinlabel $*$ at 140 10
  \endlabellist
  \includegraphics[scale=1]{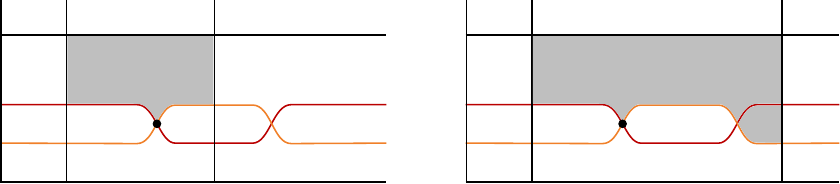}
  \caption{Two types of pentagons that contribute to the map $\pinchmap_X$. Below, we analyze how points in the various marked regions contribute to the combinatorial counts used to compute the gradings of the associated generators. }
  \label{fig:Xswap-pent}
\end{figure}

Next, we investigate how $\pinchmap_X$ interacts with the Maslov grading and the Alexander filtration. Suppose $p \in \ePent (\x, \y)$ is an empty pentagon whose domain includes the small pentagon above the intersection point $a$, as seen in \fullref{fig:combo}. After performing cyclic commutations if necessary, we may assume the $X$ swap is performed in the first and second rows, and the $O$ in the second row is in the last column; see \fullref{fig:Xswap-pent}. This allows for a simpler discussion of the gradings computation. Below, we assume that the right edge of $p$ is to the left of the second intersection point of $\alpha$ and $\alpha'$; see the left of \fullref{fig:Xswap-pent}. The other case is analogous.

Observe that if $q$ is a point in the interior of the starred regions in \fullref{fig:Xswap-pent}, then  $\J(\{q\},\x) = \J(\{q\},\y)$, where the left hand side is computed on $\G_+$ and the right hand side on $\G_-$. Next, exactly one of the regions labeled with $1$ and $2$ contains a marking, which is an $X$. Call this marking $X_1$ and note that $\J(\{X_1\},\{x_1\}) = \J(\{X_1\},\{y_1\})$, while $\J(\{X_1\},\{x_2\}) = 1/2$ and $\J(\{X_1\},\{y_2\})=0$. So $X_1$ contributes $1/2$ to the count $\J(\x, \X)-\J(\y, \X)$. Similarly, let $X_2$ be the unique marker in regions $3$ and $4$, and note that $X_2$ contributes $1/2$ to the count $\J(\x, \X)-\J(\y, \X)$. There is no marker in region $5$, whereas for any marker $q$ in the top unlabelled region, we have $\J(\{q\},\x) = \J(\{q\},\y)$. (Note that the bottom unlabelled region is in fact empty.) Finally, each point $q$ in the interior of $p$ contributes $1$ to $\J(\x, \{q\})- \J(\y, \{q\})$. 

The above paragraph implies that  $\J(\y, \X) -\J(\x, \X)=-1-\abs{\X\cap p}$ and $\J(\y, \OO) -\J(\x, \OO)=0$. Also, observe that $\J(\X, \X)$ decreases by one as we move from $\G_+$ to $\G_-$, whereas $\J(\OO, \OO)$ is unchanged, and $\J(\y, \y)-\J(\x, \x)=-1$. Thus, 
\begin{align*}
M(\y) - M(\x) &= -1,\\
A(\y) - A(\x) &= \abs{p \cap \XX} + \frac{\abs{\Legm} - \abs{\Legp} - 1}{2},
\end{align*}
where $\x$ is considered an element of the unshifted complex $\fGCt(\G_+)$. Since $\abs{p \cap \XX}\geq 0$, shifting the grading and filtration of this complex as in the statement of the lemma, we obtain a filtered chain homomorphism.
The case where $p$ is a pentagon below $a$ is analogous.

Finally, we compute the image of $\xpm (\grid_+)$ under $\pinchmap_X$. The proof of \cite[Lemma~3.3]{BLW22} observes that, if $X$'s and $O$'s are both blocked, there is a unique pentagon that carries $\xp (\grid_+)$ to $\xp (\grid_-)$, and similarly for $\xm$. We observe that, allowing $X$'s to be unblocked, we may get other pentagons, but such pentagons always contain at least one $X$ inside, meaning that the target generator necessarily belongs to a lower Alexander filtration level.
\end{proof}

\subsubsection{\texorpdfstring{$O$ swap}{O swap}}
\begin{lemma}\label{lem:oswap}
Suppose $\G_+$ is obtained from $\G_-$ by an $O$ swap. Then there exists a filtered chain homomorphism
\begin{equation*}
\pinchmap_{O} \colon \fGCt(\G_+) \to \fGCt(\G_-) \left\llbracket 1, \frac{\abs{\Legp} - \abs{\Legm} + 1}{2} \right\rrbracket
\end{equation*}
such that
\[
  \pinchmap_O(\xpm(\G_+))=\xpm(\G_-).
\]
\end{lemma}

\begin{proof}
Given $\x \in S(\G_+)$ and $\y\in S(\G_-)$, let $\Tri(\x,\y)$ be the set of 
triangles in the combined diagram whose vertices are points in $(\x\triangle 
\y) \cup \{b\}$, with the following conditions.
First, $\Tri(\x,\y)$ is empty unless $\x$ and $\y$ agree at $m-1$ points. An element $p \in \Tri (\x, \y)$ is an embedded triangle with non-reflex angles whose boundary lies on the horizontal and vertical circles (including $\alpha$ and $\alpha'$) and whose vertices are points in $(\x \triangle \y) \union \set{b}$, such that $\bdy (\bdy p \cap \betas) = \x - \y$.
Note that a triangle is automatically empty, in the sense that $\Int(p) \cap \x = \emptyset$. Let $\Tri_\OO (\x, \y)$ be the subset of $p \in \Tri(\x, \y)$ such that $t \cap \OO = \emptyset$. See \cite[Figure~4]{Won17} for some examples of triangle domains in a similar context.
Define
\[
  \pinchmap_O \colon \fGCt(\G_+)\rightarrow \fGCt(\G_-)
\]
to be the linear map that counts 
these triangles: \[ \pinchmap_O (\x)= \sum \limits_{\y \in S(\G_-)} \sum 
  \limits_{t\in \Tri_\OO(\x,\y)} \y.\]

The proof that $\pinchmap_O$ is a chain map is analogous to the proof for $\pinchmap_X$. This time, we consider concatenations of rectangles and triangles, rather than rectangles and pentagons;  see, for example, \cite[Lemma~3.4]{Won17} for details. The proof that $\pinchmap_O$, with the shifts in the statement of the lemma, respects the Maslov and Alexander filtration, is also a direct computation similar to the proof for $\pinchmap_X$. Finally, by \cite[Theorems~3.7 and 3.8]{BLW22}, we know that when $X$'s and $O$'s are both blocked, the image of $\xpm(\G_+)$ under $\pinchmap_O$ is $\xpm(\G_-)$. We now allow $X$'s to be unblocked, but in fact, no $X$'s can be in a triangle! Thus our proof is complete.
\end{proof}

\subsection{Birth moves}\label{ssec:birth}
\begin{lemma}\label{lem:birth}
Suppose $\G_+$ is obtained from $\G_-$ by a birth move. Then there exists a filtered chain homomorphism
\[
\birthmap\colon \fGCt(\G_+) \to  \fGCt(\G_-) \llbracket -1, 0 \rrbracket
\]
such that 
\[
  \birthmap(\xpm(\G_+))=\xpm(\G_-).
\]
\end{lemma}

\begin{proof}
Our strategy is to extend the birth map in \cite[Proposition~3.9]{BLW22} to allow rectangles that contain $X$'s. For completeness, we repeat the set up in \cite{BLW22} below.

Through row and column commutations, there exists a diagram such that the birth 
occurs directly to the bottom right of an $O$. Then, if we define the points 
$a$ and $b$ as shown in \fullref{fig: BirthDiag}, we can decompose $S(\G_+)$ into 
the disjoint union $\AB \sqcup \AN \sqcup \NB \sqcup \NN$ where
\begin{itemize}
    \item $\AB$ consists of $\x \in \SGp$ with $a,b\in \x$;
    \item $\AN$ consists of $\x \in \SGp$ with $a \in \x$ and $b \notin \x$;
    \item $\NB$ consists of $\x \in \SGp$ with $a \notin \x$ and $b \in \x$; and
    \item $\NN$ consists of $\x \in \SGp$ with $a,b \notin \x$.
\end{itemize}
\begin{figure}[ht]
  \vspace{.2cm}
  \labellist
  \pinlabel $O_1$ at 10 45
  \pinlabel $\beta_1$ at 20 62
  \pinlabel $\alpha_3$ at 47 38
  \pinlabel $\beta_1$ at 100 82
  \pinlabel $\beta_2$ at 120 82
  \pinlabel $\beta_3$ at 140 82
  \pinlabel $\alpha_3$ at 166 59
  \pinlabel $\alpha_2$ at 166 39
  \pinlabel $\alpha_1$ at 166 18
  \pinlabel $O_1$ at 90 68
  \pinlabel $O_2$ at 110 47
  \pinlabel $O_3$ at 132 25
  \pinlabel $X_2$ at 110 25
  \pinlabel $X_3$ at 132 47
  \pinlabel {\small{$a$}} at 103 54
  \pinlabel {\small{$b$}} at 123 32
  \endlabellist
  \includegraphics[scale=1]{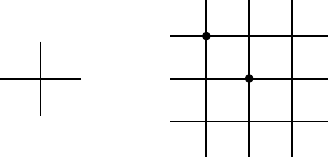}
  \caption{Left: Part of the grid diagram $\G_-$. Right: The corresponding part of the diagram $\G_+$ obtained from $\G_-$ by a birth move.}
  \label{fig: BirthDiag}
\end{figure}
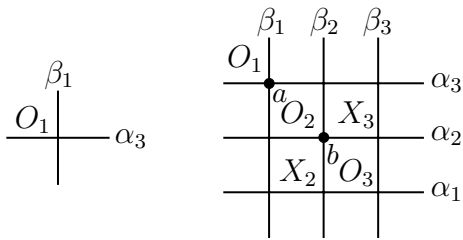

This induces a decomposition of the vector space $\fGCt(\G_+)$ as a direct sum, 
\[\fGCt(\G_+) = \widetilde{\AB} \oplus \widetilde {\AN} \oplus \widetilde{\NB} \oplus \widetilde{\NN}\]
where the summands are the subspaces generated by the corresponding subsets. Note that we have a sequence of subcomplexes
\[\NN \subset \NB \oplus \NN \subset \NN \oplus \NB \oplus \AN \subset \fGCt(\G_+),\]
since no rectangle contributing to the differential can end at $a$ or at $b$.
Let $(\widetilde{\AB},\widetilde{\delta}_{\AB})$ be the quotient complex of 
$\fGCt(G_+)$ by $\NN \oplus \NB \oplus \AN$. There is a natural bijection of 
the generators in $\AB$ and the generators in $S(\G_-)$ given by $\x\mapsto 
\x'\coloneqq \x \setminus \{a,b\}$. This map extends linearly to an isomorphism 
$e\colon \widetilde{\AB} \rightarrow \fGCt(\G_-)$ of (not necessarily filtered) 
chain complexes.

For any $\x,\y \in \AB$, there is a bijection of empty rectangles $r \in \eRect (\x, \y)$ not containing $O$'s in $\widetilde{\AB}$ and empty rectangles 
$r' \in \eRect (\x', \y')$ not containing $O$'s in $\G_-$, since any rectangle that contains the $2 \times 2$ square where the new unknot is must necessarily contain an $O$.
By \eqref{eq:maslov-gr-rel} and \eqref{eq:alex-filt-rel}, this bijection shows
that $e$ respects the Alexander filtration and is homogeneous.

Let $\OO$ and $\XX$ denote the sets of $O$ and $X$ markers in $\grid_+$ respectively. Now, for $\x\in \NB$ and $\y\in \AB$, let 
\[\Rect_{\AB}(\x,\y) \subset \Rect_{\G_+}(\x,\y)\]
be the subset of rectangles $p$ that satisfy \begin{itemize}
    \item $p \cap \OO = \{ O_2,O_3\}$
    \item $p \cap \X \supseteq \{X_2, X_3\} $
    \item $\Int(p) \cap \x=\Int(p) \cap \y =\{b\}.$
\end{itemize}
The second bullet item here is the key difference from \cite[Proposition~3.9]{BLW22}, which requires an equality instead of an inclusion.
Let $\psi$ be the linear map defined on generators by counting such rectangles: \[
\psi(\x)= \sum \limits_{\y \in \AB} \sum \limits_{p \in \Rect_{\AB}(\x,\y)} \y.\]

Let $\Pi \colon \fGCt(\G_+)\rightarrow \widetilde{\NB}$ be the projection onto 
the summand $\widetilde{\NB}$. Finally, let $\birthmap$ be the linear map 
defined by the composition: \[\birthmap= e \circ \psi \circ \Pi.\]

First, we show that $\birthmap(\xpm(\G_+))=\xpm(\G_-)$. Note that $\xpm (\G_+) 
\in \NB$, so $\Pi (\xpm (\grid_+))=\xpm (\grid_+)$. Thus, $\birthmap(\xpm 
(\grid_+))=e\circ \psi (\xpm (\grid_+))$.  As shown in \fullref{fig: BX+}, 
there is a unique rectangle of the type that defines $\psi$ that starts at 
$\xpm (\grid_+)$; composing with $e$, we see that 
$\birthmap(\xpm(\G_+))=\xpm(\G_-)$.
\begin{figure}[ht]
        \labellist
    \pinlabel $O_1$ at 6 71
    \pinlabel $O_2$ at 24 53
    \pinlabel $O_3$ at 43 34
    \pinlabel $X_2$ at 24 34
    \pinlabel $X_3$ at 43 53
    \pinlabel $X_1$ at 6 -3
    \pinlabel $X_1$ at 145 18
    \pinlabel $O_1$ at 145 51
    \pinlabel $O_1$ at 225 63
    \pinlabel $O_3$ at 262 27
    \pinlabel $O_2$ at 243 45
    \pinlabel $X_2$ at 262 45
    \pinlabel $X_3$ at 243 27
    \pinlabel $X_1$ at 298 63
    \pinlabel $O_1$ at 382 45
    \pinlabel $X_1$ at 414 45
    \pinlabel $e$ at 96 45
    \pinlabel $e$ at 335 45
    \pinlabel $\grid_+$ at 33 -15
    \pinlabel $\grid_-$ at 150 -15
    \pinlabel $\grid_+$ at 262 -15
    \pinlabel $\grid_-$ at 401 -15
    \endlabellist
    \vspace{.6cm}
  \includegraphics[scale=1]{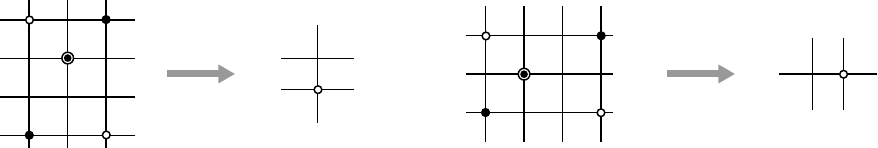}
  \vspace{.6cm}
  \caption{Left: In $\grid_+$, the generator $\xp (\grid_+)$ (solid circles) and the unique generator $\y$ (hollow circles) for which $\Rect_{\mathit{AB}} (\xp (\grid_+), \y)$ is non-empty. There is a single rectangle $p \in \Rect_{\mathit{AB}} (\xp (\grid_+), \y)$, showing that $\psi (\xp (\grid_+)) = \y$. The corresponding generator $\birthmap (\xp (\grid_+)) = \xp (\grid_-)$ is shown in $\grid_-$. Right: The analogous analysis for $\xm$.}
  \label{fig: BX+}
\end{figure}

The proof that $\birthmap$ is a chain map is, almost \emph{verbatim}, the same as that of \cite[Lemma~3.10]{BLW22}. Turning to the Maslov grading and Alexander filtration, first note that $e$ and $\Pi$ are both homogeneous with respect to both the Maslov and Alexander functions. From the definition of $\psi$, it is clear that $\psi$ is homogeneous with respect to the Maslov grading; thus, one may use \eqref{eq:xpm grad} together with the fact that $\Legp$ is the disjoint union of $\Legm$ with an unknot with $\tb = -1$ and $\rot = 0$, to compute the Maslov grading shift to be $1$. Finally, we can use \eqref{eq:alex-filt-rel} to compute the Alexander filtration shift of a rectangle $p \in \Rect_{\mathit{AB}} (\x, \y)$ to be $- \# (r \cap \XX)$. Temporarily blocking the $X$'s in $\psi$ except $X_2$ and $X_3$, \cite[Lemma~3.12]{BLW22} shows that $\birthmap$ is homogeneous with respect to, and in fact preserves, the Alexander function. Unblock the $X$'s now; combining the facts above, we see that if $\y'$ appears as a term in $\birthmap (\x)$ via a rectangle $p \in \Rect_{\mathit{AB}} (\x, \y)$, then
\[
  A (\y) - A (\x) = - \# (p \cap \XX) + 2 \leq 0,
\]
with equality when $p$ does not contain any $X$'s, which is indeed attained. This shows that $\birthmap$ is a filtered chain map with no filtration level shift.
\end{proof}

\subsection{Proof of the weak functoriality and obstruction}
\label{ssec:cob-combined}

\begin{proof}[Proof of \fullref{thm:functoriality}]
  The result is obtained by applying \fullref{lem:filt-map-ss-el} on the filtered chain homomorphisms in \fullref{lem:xswap}, \fullref{lem:oswap}, \fullref{lem:birth} to obtain the desired maps for pinches and births, and using \fullref{thm:leg} to obtain maps for Legendrian isotopies. As a decomposable Lagrangian cobordism is made up of these pieces, the associated map can be defined by composing the maps above. For more details, see \cite[proof of Theorem~1.5]{BLW22}.
\end{proof}

\begin{proof}[Proof of \fullref{thm:cob}]
  This is immediate from \fullref{thm:functoriality} and \fullref{lem:filt-map-ss-el}; in particular, the existence of the map in \fullref{thm:functoriality} shows that $\npm (\Legm) \geq \npm (\Legp)$.
\end{proof}

\begin{proof}[Proof of \fullref{cor:filling}]
  A decomposable Lagrangian filling of $\Leg$ is the concatenation of a Lagrangian birth and a decomposable Lagrangian cobordism from the undestabilizable Legendrian unknot $\Leg_U$ to $\Leg$. Noting that $\npm (\Leg_U) = \infty$ and $\lambdatpm_i (\Leg_U) \neq 0$ for all $i \in \Z_{\geq 1}$, the corollary follows.
\end{proof}

\section{Computations}
\label{sec:computation}

While our invariants have nice definitions in terms of the classes of $\xpm(\G)$ in sub-quotients of our complex, for computational purposes we would like to rephrase these definitions in terms of a particular kind of sub-quotient: the homology of a complex. Essentially, given a grid diagram $\G$ and an integer $r \ge 1$, we would like to be able to answer the questions
\begin{enumerate}
    \item Is $\lambdatpm_r(\G)$ well-defined?
    \item Is $\lambdatpm_r(\G) = 0$ for a given $r$?
\end{enumerate}
by computing the class of some element in the homology of some complex. This 
would allow us to use the techniques of \cite{NOT08} to compute the homology 
class of such an element by doing local searches and reductions to make 
computations more efficient and feasible for larger knots.

We will tackle the second question first. It turns out that the question of whether $\lambdatpm_r$ is zero or not is equivalent to the question of whether or not $\xpm$ is null-homologous in a particular sub-quotient complex.

\begin{proposition}\label{prop:nullhomologous}
Given $x \in \filt_p C$ with $\diff x \in \filt_{p-1}$ and $r \in \Z_{\ge 0}$, we have that $[x]^r = 0 \in E^r$ if and only if $[x] = 0 \in H_*(\filt_{p+r-1} C / \filt_{p-1} C)$.
\end{proposition}

\begin{proof}
To start, assume $[x]^r = 0 \in E^r$. Therefore, $[x]^r \in B^r_p$, so we have that $x + z = \diff y$, where $y \in \filt_{p+r-1} C$ and $z \in \filt_{p-1} C$. Therefore, $\diff y = x + z \equiv x \pmod{\filt_{p-1} C}$, so $\diff [y] = [x] \in \filt_{p+r-1} C / \filt_{p-1} C$, thus $[x] = 0 \in H_*(\filt_{p+r-1} C / \filt_{p-1} C)$. This proves one direction.

To prove the other implication, assume $[x] = 0 \in H_*(\filt_{p+r-1} C / \filt_{p-1} C)$. Then $x = \diff y + z$, for $y \in \filt_{p+r-1} C$ and $z \in \filt_{p-1} C$. Therefore, $x \in \filt_p C \cap \diff(\filt_{p+r-1} C) + \filt_{p-1}$, thus $[x]^r \in B^r_p$ and $[x]^r = 0 \in E^r_p$.
\end{proof}

Now, we will attempt to answer the first question in a similar way.

\begin{proposition}\label{prop:cycle}
Given $x \in \filt_p C$ with $\diff x \in \filt_{p-1}$ and  $r \in \Z_{\ge 0}$, we have that $d^r [x]^r = 0 \in E^r$ if and only if $[\diff x] = 0 \in H_*(\filt_{p-1} C / \filt_{p-r-1} C)$.
\end{proposition}

\begin{proof}
To start, assume $d^r [x]^r = 0 \in E^r$. Therefore, $[x]^{r+1} \in Z^{r+1}_p$ is defined, so we know that $\diff(x + z) = y$ for some $y \in \filt_{p-r-1} C$ and $z \in \filt_{p-1} C$. Rearranging gives us that $\diff x = \diff z + y$, so $\diff x \equiv \diff z \pmod{\filt_{p-r-1} C}$, thus $[\diff x] = [\diff z] = 0 \in H_*(\filt_{p-1} C / \filt_{p-r-1} C)$.

To prove the other direction, assume $[\diff x] = 0 \in H_*(\filt_{p-1} C / \filt_{p-r-1} C)$. This means that $\diff x = \diff z + y$ for some $z \in \filt_{p-1} C$ and $y \in \filt_{p-r-1} C$. Rearranging gives us that $\diff(x + z) = y \in \filt_{p-r-1} C$, so $[x]^{r+1} \in Z^{r+1}_p$ is defined, thus $d^r [x]^r = 0 \in E^r$.
\end{proof}

Given a grid diagram $\G$, these two techniques allow us to efficiently compute our invariants $\npm(\G)$ and $\lambdatpm_i(\G)$ as follows:
\begin{enumerate}
    \item Let $r=1$. We know that $\lambdatpm_1(\G)$ is always well-defined.
    \item\label{item:repeat} Use \fullref{prop:nullhomologous} to check if 
      $\lambdatpm_r(\G) = 0$. If so, then we are done, and we can conclude that 
      $\npm(\G)=\infty$ and $\lambdatpm_i(\G) = 0$ for $i > r$.
    \item Otherwise, use \fullref{prop:cycle} to check if $d^r \lambdatpm_r(\G) = 0$. If not, then we are done, and we can conclude that $\npm(\G) = r$ and $\lambdatpm_i(\G)$ is undefined for $i > r$.
    \item If $d^r \lambdatpm_r(\G) = 0$, then $\lambdatpm_{r+1}(\G)$ is well-defined, so we 
      may increment $r$ by $1$ and repeat the process from \eqref{item:repeat}.
\end{enumerate}
This algorithm is implemented in \cite{JPSWW22:program}, which was used to produce the results in \fullref{ssec:intro-effectiveness}.

\bibliographystyle{mwamsalphack}
\bibliography{references}

\end{document}